\documentclass[11pt, reqno]{amsart}   	% use "amsart" instead of "article" for AMSLaTeX format
\usepackage{amssymb, latexsym}
\usepackage{amsmath}
\usepackage[toc,page]{appendix}
\usepackage{braket}
\usepackage{breqn}
\usepackage{caption}
\usepackage{color}
\usepackage{esint}
\usepackage[T1]{fontenc}
\usepackage{geometry}         
\usepackage{graphicx}
\usepackage{hyperref}
\usepackage{latexsym}
\usepackage{mathrsfs}
\usepackage{mathtools}
\usepackage{subcaption}
\usepackage{graphics}

\usepackage[dvipsnames]{xcolor} %can be used to highlight text

\usepackage{times}
\usepackage{url} %does nice formatting of urls
\usepackage{tikz}

\theoremstyle{plain}
\numberwithin{equation}{section}
\newtheorem{thm}{Theorem}[section]
\newtheorem{theorem}[thm]{Theorem}
\newtheorem{lemma}[thm]{Lemma}
\newtheorem{lem}[thm]{Lemma}

\newtheorem{example}[thm]{Example}
\newtheorem{definition}[thm]{Definition}
\newtheorem{defi}[thm]{Definition}
\newtheorem{question}[thm]{Question}

\newtheorem{remark}[thm]{Remark}

\def\al{\alpha}

\def\be{\beta}

\def\N{\mathbb{N}}

\def\ex{\exists}

\newcommand\beq{\begin{equation}}
\newcommand\eeq{\end{equation}}
\newcommand\bea{\begin{eqnarray}}
\newcommand\eea{\end{eqnarray}}
\newcommand\bi{\begin{itemize}}
\newcommand\ei{\end{itemize}}
\newcommand\ben{\begin{enumerate}}
\newcommand\een{\end{enumerate}}

\address{Cluster Innovation Centre, University of Delhi, New Delhi, 110007}
\email{vedantbonde19@gmail.com}

\address{227 Ayres Hall, 1403 Circle Drive, Knoxville, TN, 37916}
\email{jsiktar@vols.utk.edu}

\title{On the Combinatorics of Placing Balls into Ordered Bins}
\author{Vedant Bonde, Joshua M. Siktar}

\begin{document}

\maketitle

\centerline{\bf Abstract} 

\noindent
In this paper, we use techniques of enumerative combinatorics to study the following problem: we count the number of ways to split $n$ balls into nonempty, ordered bins so that the most crowded bin has exactly $k$ balls. We find closed forms for three of the different cases that can arise: $k > \frac{n}{2}$, $k = \frac{n}{2}$, and when there exists $j < k$ such that $n = 2k + j$. As an immediate result of our proofs, we find a closed form for the number of positive integer solutions to $x_1 + x_2 + \dots + x_{\ell} = n$ with the attained maximum of $\{x_1, x_2, \dots, x_{\ell}\}$ being equal to $k$, when $n$ and $k$ have one of the aforementioned algebraic relationships to each other.  The problem is generalized to find a formula that enumerates the total number of ways without specific conditions on $n, \ell, k$. Subsequently, various additional identities and estimates related to this enumeration are proven and interpreted.

%\tableofcontents
\section{Introduction}\label{intro1}
One of the most elementary, well-known enumerative combinatorics problems asks:
\begin{question}[Balls into Bins]\label{ballsinbinselem}
How many ways can we split $n$ balls into $\ell$ nonempty ordered bins?
\end{question}
This is better known as the ``stars-and-bars" problem and frequents itself in combinatorics textbooks such as \cite{Bona, Mil, Wil}. The well-known formula that answers this question is ${{n - 1}\choose{\ell - 1}}$. However, there are many different restrictions on the contents of the bins that can increase the difficulty of the problem at hand. Here is one example:
\begin{question}[Balls into Bins with Minimum Capacity]\label{ballsinbinsmincapac}
How many ways can we split $n$ balls into $\ell$ nonempty ordered bins so that each bin has at least $t$ balls?
\end{question}
It turns out this is merely a generalization of Question \ref{ballsinbinselem} where $t$ can have a value other than $1$. The closed form for this problem is also well-known and is ${{n - (t - 1)\ell - 1}\choose{\ell - 1}}$. In this paper, we will address a question that sounds very similar to \ref{ballsinbinsmincapac} but is actually far more complicated:
\begin{question}[Balls into Bins with Maximum Capacity]\label{questionballsinbins} How many ways can we split $n$ balls into any number of nonempty ordered bins where the most crowded bin has exactly $k$ balls?
\end{question}
%A natural extension to Question \ref{questionballsinbins} could involve a variable number of balls and bins not specified in terms of number of bins or the restriction of balls in a bin. 
A natural variation of Question \ref{questionballsinbins} simply involves fixing the total number of bins used.
\begin{question}[Generalized Balls into Bins with restrictions problem]\label{gen_ballsbins}
Let $n, \ell, k \in \N^{+}$. This problem asks how many ways we can split $n$ balls into $\ell$ non-empty bins such that the most crowded bin has exactly $k$ balls.
\end{question}

%We will aim to find the formula for $M_{n,\ell,k}$ by finding sub-problems of the original problem and applying various combinatorics techniques to these sub-problems. 
%%I think this sentence is redundant

Question \ref{gen_ballsbins} aims to enumerate all the possible combinations of balls into bins with the restriction on the maximum number of balls in a bin. The condition of the most crowded bin having exactly $k$ balls still holds, but with the added fact that we are also given the exact number of non-empty bins to be filled. 

%The difference between this problem and Question \ref{questionballsinbins} is that this problem does not require the condition that the most crowded bin has exactly $k$ balls, but merely that no bin has more than $k$ balls. Another key aspect of Question \ref{gen_ballsbins} is that the bins in this problem may be empty. 

\begin{defi}[Generalized Bins restriction problem]\label{bins_restrict}
Let $n, k, \ell \in \N^{+}$. The Generalized Bins restriction problem aims to find the number of ways to split n balls into $\ell$ nonempty bins such that the maximum number of balls in each bin is at most $k$. 
\end{defi}
Throughout the paper we will let $B_{n, k}$ denote the answer to Question \ref{questionballsinbins} for chosen values of $n$ and $k$. Along with this, $M_{n, \ell, k}$ will denote the number of ways to split $n$ balls into exactly $\ell$ nonempty bins where the most crowded bin has exactly $k$ balls (the answer to Question \ref{gen_ballsbins}). We will also denote the quantity in Definition \ref{bins_restrict} as $R_{n, \ell, k}$. The Principle of Inclusion and Exclusion (P.I.E.) will be crucial for interpreting this class of problems for the following reason: if the maximum number of balls in a bin is \textit{exactly} $k$, then \textit{at least} one bin will have $k$ balls in it. 

Our approach to answering Question \ref{questionballsinbins} revolves around considering two cases separately: the case where most of the balls are in a single bin, and the case where the balls are, loosely speaking, ``more spread out." We formalize this notion of ``spread apart" with the following definition, which is utilized in many lemma and theorem statements throughout the paper.

\begin{definition}\label{dominantbindef} A configuration of $n$ balls into bins has a \textbf{dominant bin} if the most crowded bin has exactly $k$ balls, where $\frac{n}{2} < k < n$. Otherwise we say the configuration has no dominant bin.
\end{definition}

Now we proceed to survey the literature to review related problems. Binomial coefficients have surfaced in many problems of an enumerative nature, including combinatorial inequalities, lattice walks, and Stirling Numbers \cite{Al, Bor, Boy1, Boy2, Che, Don, Fah1, Fah2, Fan, Guo, Hard, Kol, Li, Mpg, Nim, Roy, Sha, Sofo, Sta, Sun}. There are also numerous applications of binomial coefficient identities to number theory and computer science. One prominent such application is that of the \textbf{Bernoulli trial}, where we flip a [weighted] coin in succession many times and track the number of consecutive heads. Long-time asymptotic behavior of Bernoulli trials is explored in \cite{Cha1, Cha2, Hol, Hua, Kon, Kum1, Kum2, Meg, Rig}. Furthermore, Lucas polynomials have played a role in calculating the number of ways to place the numbers $\{1, 2, \dots, n\}$  on a circle and find $r$ adjacent numbers on the circle where no $k$ of them are consecutive, and this problem is closely related to the aforementioned Bernoulli trial problem (see \cite{Cha1, Cha2}). The main difference with our problem is that the bins we place balls into are in a line, rather than a circular formation. 

We now state the four main results of this paper. The first three are all closed forms for $B_{n, k}$ when there is a different algebraic relationship between $n$ and $k$. The fourth one is a summation formula for $M_{n, \ell, k}$. One remarkable attribute of our proofs is that we do not utilize generating functions in any way, and instead resort to using more elementary binomial coefficient manipulations.

\begin{theorem}[Closed form: dominant bin]\label{B_{n, k}DomClusterMain} If $n, k \in \N^+$ with $\frac{n}{2} < k < n$ then
\begin{equation}\label{groupingBallsClosedEqMain}
B_{n, k} \ = \ (n - k + 3)2^{n - k - 2}. 
\end{equation}
\end{theorem}

\begin{theorem}[Closed form for $B_{2k, k}$] If $k \in \N^+$ then
\begin{equation}\label{B_{2r, r}closedMain}
B_{2k, k} \ = \ (k + 3)2^{k - 2} - 1.
\end{equation}
\end{theorem}

\begin{theorem}[Closed form for $B_{2k + j, k}$]\label{B_{2k+j,k}closedformMain}
Let $j, k \in \N^+$ with $k > j$. Then the number of ways to split $2k + j$ balls into nonempty bins so the most crowded bin has exactly $k$ balls is \begin{equation}\label{B_{2k+j,k}closedformMainEq}
B_{2k + j, k} \ = \ (k + j + 3)2^{k + j - 2} - (3j^2 + 19j + 18)2^{j - 4}.
\end{equation}
\end{theorem}

\begin{thm}[Formula for Generalized Balls into Bins with restrictions problem]\label{generalizedBallsIntoBinsRestrictionsII}
Suppose $n, k, \ell \in \N^+$ such that $ \ell + k - 1 \leq n \leq \ell k$. Then the following identity for $M_{n, \ell, k}$ holds:

\begin{equation} \label{genBallsBinsRestEq2}
M_{n, \ell, k} \ = \ \sum^{\ell}_{t = 0} (-1)^t  {{\ell} \choose{t}}{ \left[ {{n - tk - 1}\choose{\ell - 1}} - {{n - t(k - 1) - 1}\choose{\ell - 1}} \right]}.
\end{equation}
\end{thm}

Upon inspection of the conditions for these results, one may notice they are not exhaustive of all possible values of $n, k, \ell \in \N^+$. Aside from the trivial case where $n \leq k$, these formulas do not give a general formula for $B_{mk + j, k}$ when $m \geq 3$. We focus on the $m = 2$ case (equation \eqref{B_{2k+j,k}closedformMainEq}) because it most succinctly showcases the strategy for counting the configurations of bins when there is no dominant bin. This formula is derived from doing casework on the number of bins that contain $k$ balls, and the calculations are far more tractable when $m = 2$.

The remainder of the paper is organized as follows: in Section \ref{binDomCluster} we will prove \eqref{groupingBallsClosedEqMain}, and in Section \ref{ballsinbinsnondom} we will prove \eqref{B_{2r, r}closedMain} and \eqref{B_{2k+j,k}closedformMainEq}. In both of these sections we prove the given closed forms first by fixing the number of bins used and then later removing that restriction. Section \ref{GeneralizedBins} provides generalizations of the aforementioned cases where we can find suitable summation formulas but not closed forms. This strategy enables us to establish a connection to a well-studied counting problem for solutions to certain integer equations in Section \ref{intEqn}. The conclusion, Section \ref{conclusion}, provides some finishing remarks and possible directions for future research.  

\section{Balls into bins with a dominant bin}\label{binDomCluster}

This section will be devoted to studying the balls in bins problem (Question \ref{questionballsinbins}) where there is a dominant cluster; that is, how many ways can we sort $n$ balls into bins when one bin has more than half the total number of balls. It turns out this is the most straightforward of the three cases, for the following reason: if the most crowded bin, henceforth called the \textbf{dominant bin}, has $k$ balls, then no other bin has $k$ or more balls. This simplifies the combinatorial analysis to come.
 
The first step will be to find a formula for $M_{n, \ell, k}$ when $\frac{n}{2} < k < n$. From here we will find a summation formula for $B_{n, k}$, and then find a closed form for the sum. We quickly remark that the smallest possible number of bins to be used is $2$ (one with $k$ balls, and the other with $n - k$ balls), and the largest possible number of bins is $n - k + 1$ (one with $k$ balls, and $n - k$ bins each with $1$ ball). Recall that $B_{n, k}$ denotes the number of ways to split $n$ balls into any number of ordered nonempty bins where the most crowded bin has $k$ balls. Furthermore, $M_{n, \ell, k}$ denotes the number of ways to split $n$ balls into exactly $\ell$ nonempty bins where the largest group is of size $k$. 

\begin{lemma} Let $n, \ell, k \in \N^+$ such that $\frac{n}{2} < k < n$ and $2 \leq \ell \leq n - k + 1$. Then the number of ways to split $n$ balls into $\ell$ nonempty bins where the most crowded bin has exactly $k$ balls is
\begin{equation}\label{domClusterBinPartition}
M_{n, \ell, k} \ = \ \ell{{n - k - 1}\choose{\ell - 2}}.
\end{equation}
\end{lemma}

\begin{proof} If the largest group is of size $k$, then our task reduces to splitting $n- k$ balls into nonempty groups of size at most $n - k$. Since $\frac{n}{2} < k$, this means $n - k < k$, and that all of the remaining groups have size smaller than $k$. In other words, we need not impose further restrictions when breaking the $n - k$ remaining balls into groups. This reasoning also illustrates that our proof does not generalize to any cases where $0 < k \leq \frac{n}{2}$.

We evaluate the given quantity by splitting the $n - k$ remaining balls into $\ell - 1$ bins. By the classical stars-and-bars argument depicted in \cite{MilWan}, this step can be done in ${{n - k  - 1} \choose {\ell - 2}}$ ways. Finally, we must insert the bin of size $k$ in between the $\ell - 1$ groups already established, and this can be done in $\ell$ ways; the bin of size $k$ can go between two of the other bins, or it can be put at either end of the line of bins. Multiplying the number of ways to perform these two steps together yields the desired result.
\end{proof}

Now we derive a closed form for $B_{n, k}$ when there is a dominant bin.

\begin{theorem}[Closed form: dominant bin]\label{domBinClosedForm} If $n, k \in \N^+$ with $\frac{n}{2} < k < n$ then
\begin{equation}\label{groupingBallsClosedEq}
B_{n, k} \ = \ (n - k + 3)2^{n - k - 2}. 
\end{equation}
\end{theorem}

\begin{proof} We first obtain a summation formula for $B_{n, k}$ by summing \eqref{domClusterBinPartition} over $2 \leq \ell \leq n - k + 1$:
\begin{equation}\label{groupingBallsEq1}
B_{n, k} \ = \ \sum^{n - k + 1}_{\ell = 2}\ell{{n - k - 1} \choose {\ell - 2}}.
\end{equation}
Now we can verify the proposed closed form for $B_{n, k}$. Let $t := n - k$, because we will be using $t$ as a parameter. Then
\begin{equation}\label{groupingBallsClosedEq1}
B_{n, k} \ = \ \sum^{t}_{\ell = 1}(\ell + 1){{t - 1} \choose {\ell - 1}}.
\end{equation}
It remains to evaluate the sum in \eqref{groupingBallsClosedEq1}, which fortunately is relatively easy once we realize that the sum resembles the formula \eqref{BCIdentitym=1} with different variable labels. In particular, that formula gives us
\begin{equation}\label{groupingBallsClosedEq2}
\sum^{t - 1}_{\ell = 0}\ell{{t - 1}\choose{\ell}} \ = \ (t - 1)2^{t - 2}.
\end{equation}
Shifting the index of the sum in \eqref{groupingBallsClosedEq2} and adding $2\sum^{t}_{\ell = 1}{{t - 1}\choose{\ell - 1}}$ to both sides of the equation gives
\begin{equation}\label{groupingBallsClosedEq3}
\sum^{t}_{\ell = 1}(\ell + 1){{t - 1}\choose{\ell - 1}} \ = \ (t - 1)2^{t - 2} + 2\sum^{t}_{\ell = 1}{{t - 1}\choose{\ell - 1}}.
\end{equation}
Upon shifting the index of the sum on the right-hand side of \eqref{groupingBallsClosedEq3} by $1$ we notice it equals $2^{t - 1},$ and so
\begin{equation}\label{groupingBallsClosedEq4}
B_{n, k} \ = \ (t + 3)2^{t - 2}.
\end{equation}
The desired result follows upon substituting $n  - k$ back in place of $t$.
\end{proof}

\begin{remark} This lemma does not hold for $n = k$ because the second step becomes degenerate if all of the balls are in a single group of size $k$. It is obvious that $B_{k, k} = 1$ and we will henceforth ignore this case.
\end{remark}

\begin{remark} It makes heuristic sense that this formula depends on $n - k$ but not on $n$ nor $k$ individually. This is because the first step in our proof reduces the group of balls to a group of size $n - k$.
\end{remark}

\section{Balls in bins without a dominant bin}\label{ballsinbinsnondom}

In Section \ref{binDomCluster}, our focus was on the case where most (more than half) of all balls were placed in the same bin. In this section we instead assume the sunk pool balls are more spread out, and it turns out most of the theory developed for the dominant cluster case is no longer valid. 

We will begin with a subsection devoted to the special case where $n = 2k$. This situation highlights the main distinction between the behavior of the case where there is a dominant bin and when there is not, while still having a closed form very similar to that of the dominant bin case described in Section \ref{binDomCluster}. Afterward, we will derive a closed form for $B_{2k + j, k}$ when $0 < j < k$, which will involve a combination of several sums that, while manageable with purely elementary methods, is non-trivial. The general approach in both subsections will be to derive a formula for $M_{2k + j, \ell, k}$ and use that to derive a formula for $B_{2k + j, k}$; in the first subsection $j$ will equal zero, and in the second subsection $j$ will be positive but less than $k$.

\subsection{Formula for $B_{2k, k}$}\label{B_{2k,k}}

\begin{lemma} If $k \in \N^+$ then the number of ways to split $2k$ balls into $\ell$ nonempty bins where $2 \leq \ell \leq k + 1$ and the most crowded bin has exactly $k$ balls is
\begin{equation}\label{B_{2k, k}FixedBin}
M_{2k, \ell, k} \ = \ \begin{cases}
1, \ \ \ \ \ \ \ \ \ \ \ell = 2 \\
\ell{{k - 1}\choose{\ell - 2}}, \ 3 \leq \ell \leq k + 1 
\end{cases}.
\end{equation}
\end{lemma}

\begin{proof}

In the special case that $\ell = 2$ there is only one possible configuration: two bins, each containing exactly $k$ balls. The existence of this special case is contrary to our derivation of a formula for $B_{2k + j, k}$ for $0 < j < k$ as we will need at least $3$ bins when $j > 0$. On the other hand, if there are at least three bins we cannot possibly have two bins each containing $k$ balls.

Now we assume $\ell \geq 3$. Since the largest possible number of balls in a bin is $k$, and we only have one bin with $k$ balls, each of these additional bins can have at most $k - 1$ balls. Moreover, we cannot have more than $k$ bins besides the one containing exactly $k$ balls (the ``greedy" way to do this is to put $k$ balls in the first bin and a single ball in each bin thereafter). Thus once we fix the number of bins, we want to count the number of ways to split $k$ balls between $\ell - 1$ bins. By the classical stars-and-bars argument detailed in \cite{MilWan}, this can be done in ${{k - 1}\choose{\ell - 2}}$ ways.

Finally, we must choose where to insert the bin with $k$ balls amongst the other $\ell$ bins. Since the bins are ordered we can do this in $\ell$ ways; multiplying the number of ways to perform each of these two steps yields the desired result.

\end{proof}

The closed form for $B_{2k, k}$ now follows readily.

\begin{theorem}[Closed Form for $B_{2k, k}$]\label{closedFormB_{2k, k}} If $k \in \N^+$ then
\begin{equation}\label{B_{2r, r}closed}
B_{2k, k} \ = \ (k + 3)2^{k - 2} - 1.
\end{equation}
\end{theorem}

\begin{proof} We can easily deduce a summation formula for $B_{2k, k}$ by summing the formula \eqref{B_{2k, k}FixedBin} for all $2 \leq \ell \leq k + 1$. We obtain
\begin{equation}\label{B_{2k, k}sum}
B_{2k, k} \ = \ 1 + \sum^{k + 1}_{\ell = 3}\ell{{k - 1}\choose{\ell - 2}} \ = \ 1 + \sum^{k - 1}_{\ell = 1}(\ell + 2){{k - 1}\choose{\ell}}.
\end{equation}
The evaluation of this sum follows from \eqref{groupingBallsClosedEq2} and \eqref{groupingBallsClosedEq3}; then the desired result is immediate.\footnote{The sequence $\{B_{2k, k}\}^{\infty}_{k = 1}$ appears in the Online Encyclopedia of Integer Sequences \cite{Slo}.}
\end{proof}

\subsection{Formula for $B_{2k + j, k}$}\label{B_{2k+j,k}closedFormSection}

In this subsection we will derive a closed form for $B_{2k + j, k}$ when $k > j$. The strategy is much more involved than the one used in Section \ref{B_{2k,k}}, and proceeds as follows. We will see that a valid configuration of balls will have either one or two bins with exactly $k$ balls; thus we calculate the number of ways to split $2k + j$ balls into a fixed number of bins so that there is at least one bin with $k$ balls, and then when there are two bins with $k$ balls each. However, it is possible to have a configuration where one bin has $k$ balls and another bin has $k + i$ balls, for some $1 \leq i \leq j$; we do not want to include these configurations in our final formula because the largest bin has more than $k$ balls in this case. The process of adding and subtracting these formulas from each other will closely resemble the Principle of Inclusion and Exclusion.

Before proceeding to the lemmas and their proofs, we introduce some notation that will be used only in this subsection.
\begin{enumerate}
\item{First, $T_{2k + j, [k + i, k]}$ denotes the number of ways to sort $2k + j$ balls into nonempty bins so that one bin has $k$ balls and another has $k + i$ balls. Every time this notation is used, we will have $1 \leq i \leq j$.}
\item{If we want the aforementioned quantity where exactly $\ell$ nonempty bins are used, we denote it as $U_{2k + j, \ell, [k + i, k]}$.}
\item{Next, $F_{2k + j, k, t}$ denotes the total number of ways to split $2k + j$ balls into nonempty bins so that at least $t$ bins have exactly $k$ balls. Every time this notation is used, $t$ will equal $1$ or $2$.}
\item{Finally, if we want the aforementioned quantity where exactly $\ell$ nonempty bins are used, we denote it as $G_{2k + j, \ell, k, t}$.}
\end{enumerate}

\begin{lemma}\label{T_{2k+j,[k+i,k]}} Let $i, j, k \in \N^+$ with $k > j > i$. Then the number of ways to sort $2k + j$ balls into nonempty bins so one bin has $k$ balls and another has $k + i$ balls is
\begin{equation}\label{B2k+i,kOvercount}
T_{2k + j, [k + i, k]} \ = \ \sum^{j - i}_{\ell = 1}(\ell^2 + 3\ell + 2){{j - i - 1}\choose{\ell - 1}}.
\end{equation}
\end{lemma}

\begin{proof} Let $1 \leq i < j$ be arbitrary. Since one bin has $k$ balls and another bin has $k + i$ balls, we have $j - i$ balls to split between nonempty bins. Say that this number of bins is denoted by $\ell$, and then $1 \leq \ell \leq j - i$. By the classical stars-and-bars method, $j - i$ balls can be split into $\ell$ nonempty bins in ${{j - i - 1}\choose{\ell - 1}}$ ways. Now, we insert the remaining two bins in between these $\ell$ bins. There are $\ell + 1$ slots in which to insert the bin with $k$ balls, and then there are $\ell + 2$ slots to insert the bin with $k + i$ balls (these last two bins are distinguishable since $i > 0$). So, there are a total of $(\ell^2 + 3\ell + 2){{j - i - 1}\choose{\ell - 1}}$ arrangements of these bins. Summing over all possible values of $\ell$ yields the desired result.
\end{proof}

\begin{remark} The other quantity related to those calculated in Lemma \ref{T_{2k+j,[k+i,k]}} that we will need is $T_{2k + j, [k + j, k]}$. However, it is clear that
\begin{equation}\label{B2k+j,k,k+j}
T_{2k + j, [k + j, k]} \ = \ 2,
\end{equation}
because the only valid configurations are those with a bin having $k$ balls and the other bin having $k + j$ balls, in either order; there must be exactly two bins.

\end{remark}

\begin{lemma}\label{B2k+j,k2kbinLem} Let $j, k \in \N^+$ with $k > j$. Then the total number of ways to split $2k + j$ balls into nonempty bins so that at least two bins have exactly $k$ balls is
\begin{equation}\label{B2k+j,k2kbin}
F_{2k + j, k, 2} \ = \ \sum^{j}_{\ell = 1}\frac{\ell^2 + 3\ell + 2}{2}{{j - 1}\choose{\ell - 1}}.
\end{equation}
\end{lemma}

\begin{proof} First notice that we cannot have more than two bins with $k$ balls each while the total number of balls is $2k + j$, because $k > j$. Thus the quantity we want to calculate is the total number of ways to split $2k + j$ balls into nonempty bins so that exactly two bins have exactly $k$ balls. In this case, the remaining $j$ balls can be split into $\ell$ nonempty bins, where $1 \leq \ell \leq j.$ By the classical stars-and-bars method, this step can be done in ${{j - 1}\choose{\ell - 1}}$ ways. Now we insert the two bins with $k$ balls in between the other bins, and there are $\ell + 1$ slots in which to do this. If both bins with $k$ balls are in the same slot, there are $\ell + 1$ ways to do this, and if the two bins with $k$ balls are in different slots, there are $\frac{\ell(\ell + 1)}{2}$ ways to do this. Thus the total number of ways to insert the two bins with $k$ balls is
\begin{equation}\label{B2k+j,k2kbinEq1}
\ell + 1 + \frac{\ell(\ell + 1)}{2} \ = \ \frac{\ell^2 + 3\ell + 2}{2}.
\end{equation}
Since the value of $\ell$ ranges from $1$ to $j$, the desired result follows.
\end{proof}
The main difference between Lemma \ref{B2k+j,k2kbinLem} and Lemma \ref{B2k+j,k1kbinLem} is that only the former precludes the possibility of having a bin with more than $k$ balls. If two of the bins each have at least $k$ balls, and there are a total of $2k + j$ balls, then no other bin can have $k$ or more balls because $k > j$. 
\begin{lemma}\label{B2k+j,k1kbinLem} Let $j, k \in \N^+$ with $k > j$. Then the total number of ways to split $2k + j$ balls into nonempty bins so that at least one bin has exactly $k$ balls is
\begin{equation}\label{B2k+j,k1kbin}
F_{2k + j, k, 1} \ = \ \sum^{k + j}_{\ell = 1}(\ell + 1){{k + j - 1}\choose{\ell - 1}} - \sum^{j}_{\ell = 1}\frac{\ell^2 + 3\ell + 2}{2}{{j - 1}\choose{\ell - 1}}.
\end{equation}
\end{lemma}
\begin{proof} Much as in the proof of \eqref{B2k+j,k2kbin}, there will be either one or two bins with exactly $k$ balls. We will find the number of ways in which to split $2k + j$ balls into nonempty bins so that at least one bin has exactly $k$ balls, that happens to count the configurations having two bins with $k$ balls twice. Deliberately allowing for this over-counting gives us an easier way to count all of the configurations with exactly one bin with $k$ balls; we will then subtract the over-counted amount, which is represented by the formula \eqref{B2k+j,k2kbin}.
Fix one bin to have $k$ balls, and suppose we are splitting the remaining $k + j$ balls into $\ell$ nonempty bins, where $1 \leq \ell \leq k + j$. We can split $k + j$ balls into $\ell$ nonempty bins in ${{k + j - 1}\choose{\ell - 1}}$ ways (by the classical stars-and-bars argument). Then, we insert the bin with $k$ balls in between the other $\ell$ bins, which can be done in $\ell + 1$ ways. Finally, sum over the possible values of $\ell$ and we obtain the first sum in \eqref{B2k+j,k1kbin}.

However, we have double-counted configurations that have two bins with exactly $k$ balls each. This is because when we insert the bin with $k$ balls, there is the possibility of having one other bin with $k$ balls already, and these two bins are not distinguishable (this reasoning is actually an implicit use of the Inclusion-Exclusion Principle). Hence we subtract the formula \eqref{B2k+j,k2kbin} and obtain the desired result.
\end{proof}

In Lemma \ref{B2k+j,k1kbinLem}, the argument is only valid because the maximum number of bins with exactly $k$ balls is two. The resulting summations become considerably more complicated if three or more bins can all have $k$ balls apiece. 

\begin{lemma}\label{k_k+i_fixedBinCount}Let $i, j, k, \ell \in \N^+$ with $i < j < k$ and $3 \leq \ell \leq j - i + 2$. Then the number of ways to split $2k + j$ balls into $\ell$ nonempty bins so one bin has $k$ balls and another has $k + i$ balls is
\begin{equation}\label{B2k+i,kOvercountFixedBins}
U_{2k + j, \ell, [k + i, k]} \ = \ (\ell^2 - \ell){{j - i - 1}\choose{\ell - 3}}.
\end{equation}
\end{lemma}

\begin{proof} 
Let $1 \leq i < j$ be arbitrary. Since one bin has $k$ balls and another bin has $k + i$ balls, we have $j - i$ balls to split between $\ell - 2$ nonempty bins. The range of possible values of $\ell$ goes from $3$ to $j - i + 2$, because one bin has $k$ balls and another has $k + i$ balls. However, we still have $j - i$ balls left over, which can be split between as few as one additional bin and as many as $j - i$ additional bins.

By the classical stars-and-bars method of \cite{MilWan}, $j - i$ balls can be split into $\ell - 2$ nonempty bins in ${{j - i - 1}\choose{\ell - 3}}$ ways. Now, we insert the remaining two bins in between these $\ell$ bins. There are $\ell - 1$ slots to insert the bin with $k$ balls, and then there are $\ell$ slots to insert the bin with $k + i$ balls (these last two bins are distinguishable since $i > 0$). So, there are a total of $(\ell^2 - \ell){{j - i - 1}\choose{\ell - 3}}$ arrangements of these bins, as desired.
\end{proof}

\begin{remark}
As a slight extension of the result of Lemma \ref{k_k+i_fixedBinCount}, we consider how to split $2k + j$ balls into exactly $\ell = 2$ nonempty bins when one bin has exactly $k$ balls. In this case, the other bin must have $k + j$ balls. Thus we cannot split $2k + j$ balls into exactly $2$ nonempty bins when the most crowded bin has exactly $k$ balls.
\end{remark}

\begin{lemma}\label{B2k+j,k2kbinLemFixedBins} Let $j, k, \ell \in \N^+$ with $j < k$ and $3 \leq \ell \leq j + 2$. Then the total number of ways to split $2k + j$ balls into $\ell$ nonempty bins so that at least two bins have exactly $k$ balls is
\begin{equation}\label{B2k+j,k2kbinFixedBins}
G_{2k + j, \ell, k, 2} \ = \ \frac{\ell^2 - \ell}{2}{{j - 1}\choose{\ell - 3}}.
\end{equation}
\end{lemma}

\begin{proof} 
First notice that we cannot have more than two bins with $k$ balls each while the total number of balls is $2k + j$, because $k > j$. Thus the quantity we want to calculate is the total number of ways to split $2k + j$ balls into $\ell$ nonempty bins so that exactly two bins have exactly $k$ balls. In this case, the remaining $j$ balls can be split into $\ell - 2$ nonempty bins, and now it is clear the total number of bins must range from $3$ to $j + 2$.

By the classical stars-and-bars method used in \cite{MilWan}, this step can be done in ${{j - 1}\choose{\ell - 3}}$ ways. Now we insert the two bins with $k$ balls in between the other bins, and there are $\ell - 1$ slots in which to place the bins with exactly $k$ balls. If both bins with $k$ balls are in the same slot, there are $\ell - 1$ ways to do this, and if the two bins with $k$ balls are in different slots, there are $\frac{(\ell - 2)(\ell - 1)}{2}$ ways to do this. Thus the total number of ways to insert the two bins with $k$ balls is
\begin{equation}\label{B2k+j,k2kbinEq1A}
\ell - 1 + \frac{(\ell - 2)(\ell - 1)}{2} \ = \ \frac{\ell^2 - \ell}{2}.
\end{equation}
Multiplying the number of ways to perform the two steps yields the desired result.
\end{proof}

In this lemma we finally combine the previous lemmas in this subsection to obtain a closed form for $M_{2k + j, \ell, k}$.

\begin{lemma}\label{split_B_{2k+j,k}_fixedBins} Let $j, k \in \N^+$ with $j < k$ and $2 \leq \ell \leq k + j + 1$. Then the total number of ways to split $2k + j$ balls into $\ell$ nonempty bins so the most crowded bin has exactly $k$ balls is
\begin{multline}\label{totalFixedBin}
M_{2k + j, \ell, k} \ = \ \begin{cases}
0, \ \ \ \ \ \ \ \ \ \ \ \ \ \ \ \ \ \ \ \ \ \ \ \ \ \ \ \ \ \ \ \ \ \ \ \ \ \ \ \ \ \ \ \ \ \ \ \ \ \ \ \ \ \ \ \ \ \ \ \ \ \ \ \ \ell = 2 \\
\ell{{k + j - 1}\choose{\ell - 2}} - \frac{\ell^2 - \ell}{2}{{j - 1}\choose{\ell - 3}} - \sum^{s}_{i = 1}(\ell^2 - \ell){{j - i - 1}\choose{\ell - 3}}, \quad \\ \ \ \ \ \ \ \ \ \ \ \ \ \ \ \ \ \ \ \ \ \ \ \ \ \ \ \ \ \ \ \ \ \ \ \ \ \ \ \ \ \ \ \ \ \ \ \ \ \ \ \ \ \ \ \ \ \ \ \ \ \ \ \ \ \ \ \ \ \ell = j + 2 - s, 1 \leq s < j \\
\ell{{k + j - 1}\choose{\ell - 2}} - \frac{\ell^2 - \ell}{2}{{j - 1}\choose{\ell - 3}}, \ \  \ \ \ \ \ \ \ \ \ \ \ \ \ \ \  \  \ \ \ \ \ \ \ \ \ \ \ \ell = j + 2 \\
\ell{{k + j - 1}\choose{\ell - 2}} , \ \ \ \ \ \ \ \ \ \  \ \ \ \ \ \ \ \ \ \ \ \ \ \ \ \ \ \ \ \ \ \ \ \ \ \ \ \ \ \ \  \ \ \ \ \ \ \ \ \ \ j + 3 \leq \ell \leq k + j + 1 
\end{cases}.
\end{multline}
\end{lemma}

\begin{proof} We will handle each case separately. If $\ell = 2$ and one bin has exactly $k$ balls, then as discussed in the proof of Lemma \ref{B2k+j,k2kbinLemFixedBins}, the other bin has $k + j$ balls, so this case gives an answer of zero.

Now we will discuss the case where $j + 3 \leq \ell \leq k + j + 1$. As stated in \eqref{B2k+j,k1kbin}, the number of ways to split $2k + j$ balls into $\ell$ nonempty bins so at least one bin has exactly $k$ balls is $\ell{{k + j - 1}\choose{\ell - 2}}$. We do not want to include configurations containing a bin with $k + i$ balls for some $1 \leq i \leq j$ but there are no such configurations when $\ell$ is at least $j + 3$ because we cannot split $j$ balls into more than $j$ nonempty bins. 

Handing the case of $3 \leq \ell \leq j + 1$ (the second case) is more subtle. Just as in the $j + 3 \leq \ell \leq k + j + 1$ case, the number of ways to split $2k + j$ balls into $\ell$ nonempty bins so at least one bin has exactly $k$ balls is $\ell{{k + j - 1}\choose{\ell - 2}}$, but there are now some configurations where one bin has $k + i$ balls for some $1 \leq i \leq j$. There will exist $1 \leq s < j$ such that $\ell = j + 2 - s$. The value of $s$ imposes a further restriction on $i$, and we will determine what term to subtract from $\ell{{k + j - 1}\choose{\ell - 2}}$ by repeatedly using Lemma \ref{k_k+i_fixedBinCount}. It is possible to have a configuration with both a bin having $k$ balls and a bin having $k + i$ balls for any $1 \leq i \leq s$ so we must subtract \eqref{B2k+i,kOvercountFixedBins} for each $1 \leq i \leq s$, proving the proposed formula in this case. Notice that since $j < k$ we cannot have multiple bins each having more than $k$ balls and a bin having exactly $k$ balls. This is why our procedure is only effective for studying $B_{2k + j, k}$ and not $B_{mk + j, k}$ for $m \geq 3$.

Finally, the case of $\ell = j + 2$ serves as an outlier. There are no configurations having both a bin with $k$ balls and a bin with $k + i$ balls for some $1 \leq i < j$, but according to \eqref{B2k+j,k2kbinFixedBins} we must subtract $\frac{\ell^2 - \ell}{2}{{j - 1}\choose{\ell - 3}}$ to get the desired result in this case.
\end{proof}

The formula \eqref{totalFixedBin} gives us a natural partition for the number of ways to split $2k + j$ balls into nonempty bins where the most crowded bin has $k$ balls: we partition on the total number of bins used.

\begin{lemma}[Summation Formula for $B_{2k + j, k}$]\label{B_{2k+j,k}sumformbasedonbinpartitionlemma} Let $j, k \in \N^+$ with $k > j$. Then the total number of ways to split $2k + j$ balls into nonempty bins so the most crowded bin has exactly $k$ balls is
\begin{equation}\label{B_{2k+j,k}sumformbasedonbinpartition}
B_{2k + j, k} = \sum^{k + j + 1}_{\ell = 2}\ell{{k + j - 1}\choose{\ell - 2}} - \sum^{j + 2}_{\ell = 3}\frac{\ell^2 - \ell}{2}{{j - 1}\choose{\ell - 3}} -\sum^{j - 1}_{i = 1}\sum^{j - i + 2}_{\ell = 3}(\ell^2 - \ell){{j - i - 1}\choose{\ell - 3}} - 2.
\end{equation}
\end{lemma}

\begin{proof} This lemma follows from summing the result of \eqref{totalFixedBin} over $3 \leq \ell \leq k + j + 1$. The double sum arises from summing ${{j - i - 1}\choose{\ell - 3}}$ over different values of $i$ that represented over-counted configurations for each value of $s$. Here $i$ ranges from $1$ to $j - 1$ because those were the possible values of $s$ in \eqref{totalFixedBin}, and we can interchange these two sums freely as they are both finite; it will be easier to evaluate this sum by summing over $\ell$ first. 

In this formula, we extend the range of the first sum to $\ell = 2$, and clearly that term will always have value $2$; hence we subtract a $2$ at the end of the equation. We do this because it will make the future calculation of that sum slightly less cumbersome.
\end{proof}

We now have a combination of sums which together represent the value of $B_{2k + j, k}$. These sums can all be converted to closed forms, but doing so is a lengthy process. We will treat the evaluation of the sums in \eqref{B_{2k+j,k}sumformbasedonbinpartition} as distinct parts of the same lemma, and then the overall evaluation of \eqref{B_{2k+j,k}sumformbasedonbinpartition} as a single closed form will be the main theorem for this section.

\begin{lemma}\label{B2k+j,ksumevallemma} Let $j, k \in \N^+$ with $k > j$. Then the sums appearing in \eqref{B_{2k+j,k}sumformbasedonbinpartition} are evaluated as follows:
\begin{equation}\label{B2k+j,ksumeval1}
\sum^{k + j + 1}_{\ell = 2}\ell{{k + j - 1}\choose{\ell - 2}} \ = \ (k + j + 3)2^{k + j - 2}
\end{equation}
\begin{equation}\label{B2k+j,ksumeval2}
\sum^{j + 2}_{\ell = 3}\frac{\ell^2 - \ell}{2}{{j - 1}\choose{\ell - 3}} \ = \ 2^{j - 4}(j^2 + 9j + 14)
\end{equation}
\begin{equation}\label{B2k+j,ksumeval3}
\sum^{j - 1}_{i = 1}\sum^{j - i + 2}_{\ell = 3}(\ell^2 - \ell){{j - i - 1}\choose{\ell - 3}} \ = \ 2^{j - 3}(j^2 + 5j + 2) - 2.
\end{equation}
\end{lemma}

\begin{proof} We will start by proving \eqref{B2k+j,ksumeval1}. In order to use the lemmas in Appendix \ref{BCID}, we shift the index of the sum by $2$:
\begin{equation}\label{B2k+j,ksumeval1EqA}
\sum^{k + j + 1}_{\ell = 2}\ell{{k + j - 1}\choose{\ell - 2}} \ = \ \sum^{k + j - 1}_{\ell = 0}(\ell + 2){{k + j - 1}\choose{\ell}}.
\end{equation}
Now split this into two sums: \eqref{B2k+j,ksumeval1EqA} equals
\begin{equation}\label{B2k+j,ksumeval1EqB}
\sum^{k + j - 1}_{\ell = 0}\ell{{k + j - 1}\choose{\ell}} + 2\sum^{k + j - 1}_{\ell = 0}{{k + j - 1}\choose{\ell}}.
\end{equation}
To handle the first sum, we use \eqref{BCIdentitym=1} from Appendix \ref{BCID}, and the second is the binomial expansion of $(1 + 1)^{k + j - 1}$, so \eqref{B2k+j,ksumeval1EqB} equals
\begin{equation}\label{B2k+j,ksumeval1EqC}
(k + j - 1)2^{k + j - 2} + 2^{k + j} \ = \ (k + j + 3)2^{k + j - 2},
\end{equation}
completing the proof of \eqref{B2k+j,ksumeval1}.

Now we will proceed with proving \eqref{B2k+j,ksumeval2}, starting by shifting the index of the sum by $3$:
\begin{equation}\label{B2k+j,ksumeval2EqA}
\sum^{j + 2}_{\ell = 3}\frac{\ell^2 - \ell}{2}{{j - 1}\choose{\ell - 3}} \ = \ \frac{1}{2}\sum^{j - 1}_{\ell = 0}(\ell^2 + 5\ell + 6){{j - 1}\choose{\ell}}.
\end{equation}
We split this sum into three based on the power of $\ell$ before the binomial coefficient:
\begin{equation}\label{B2k+j,ksumeval2EqB}
\frac{1}{2}\sum^{j - 1}_{\ell = 0}\ell^2{{j - 1}\choose{\ell}} + \frac{5}{2}\sum^{j - 1}_{\ell = 0}\ell{{j - 1}\choose{\ell}} + 3\sum^{j - 1}_{\ell = 0}{{j - 1}\choose{\ell}}.
\end{equation}
The three sums can be evaluated respectively as follows: for the first, use \eqref{BCIdentitym=2}; for the second use \eqref{BCIdentitym=1}; and finally, the third is the binomial expansion of $(1 + 1)^{j - 1}$, so \eqref{B2k+j,ksumeval2EqB} equals
\begin{equation}\label{B2k+j,ksumeval2EqC}
\frac{1}{2}j(j - 1)2^{j - 3} + \frac{5}{2}(j - 1)2^{j - 2} + 3 \cdot 2^{j - 1}.
\end{equation}
We can factor out a $2^{j - 4}$ from every term and simplify the resulting quadratic factor to obtain \eqref{B2k+j,ksumeval2}.

Finally, we will prove \eqref{B2k+j,ksumeval3}. Again we start by shifting the index of the [inner] sum by $3$:
\begin{equation}\label{B2k+j,ksumeval3EqA}
\sum^{j - 1}_{i = 1}\sum^{j - i + 2}_{\ell = 3}(\ell^2 - \ell){{j - i - 1}\choose{\ell - 3}} \ = \  \sum^{j - 1}_{i = 1}\sum^{j - i - 1}_{\ell = 0}(\ell^2 + 5\ell + 6){{j - i - 1}\choose{\ell}}.
\end{equation}
Break \eqref{B2k+j,ksumeval3EqA} into three sums based on the power of $\ell$ to allow us to use the identities in Appendix \ref{BCID}. Then, \eqref{B2k+j,ksumeval3EqA} equals
\begin{equation}\label{B2k+j,ksumeval3EqB}
\sum^{j - 1}_{i = 1}\sum^{j - i - 1}_{\ell = 0}\ell^2{{j - i - 1}\choose{\ell}} + 5\sum^{j - 1}_{i = 1}\sum^{j - i - 1}_{\ell = 0}\ell{{j - i - 1}\choose{\ell}} + 6\sum^{j - 1}_{i = 1}\sum^{j - i - 1}_{\ell = 0}{{j - i - 1}\choose{\ell}}.
\end{equation}
Now, we use \eqref{BCIdentitym=2} on the first double sum and \eqref{BCIdentitym=1} on the second double sum, and notice the third term's inner sum is the binomial expansion of $(1 + 1)^{j - i - 1},$ allowing us to conclude that \eqref{B2k+j,ksumeval3EqB} equals
\begin{equation}\label{B2k+j,ksumeval3EqC}
\sum^{j - 1}_{i = 1}((j - i - 1)(j - i)2^{j - i - 3} + 5(j - i - 1)2^{j - i - 2} + 6 \cdot 2^{j - i - 1}).
\end{equation}
To greatly simplify the remainder of this calculation, we use a trick that somewhat resembles u-substitution for integrals. The current sum ranges from $i = 1$ to $i = j - 1$, where $j$ is fixed, so if we let $t = j - i$ and index the sum with respect to $t$, the range of the sum is $t = j - 1$ to $t = 1$. This greatly simplifies the summand, and by performing this substitution we realize \eqref{B2k+j,ksumeval3EqC} equals
\begin{equation}\label{B2k+j,ksumeval3EqD}
\sum^{j - 1}_{t = 1}(t(t - 1)2^{t - 3} + 5(t - 1)2^{t - 2} + 6 \cdot 2^{t - 1}).
\end{equation}
We utilize a series of algebraic steps that are not particularly insightful but lead us to a closed form for \eqref{B2k+j,ksumeval3EqD}:
\begin{eqnarray}
\sum^{j - 1}_{t = 1}(t(t - 1)2^{t - 3} + 5(t - 1)2^{t - 2} + 6 \cdot 2^{t - 1}) \ = \ \frac{1}{8}\sum^{j - 1}_{t = 1}2^t(t^2 + 9t + 14) \ &= \  \nonumber \\
 \frac{1}{8}\left(\sum^{j - 1}_{t = 1}t^22^t + 9\sum^{j - 1}_{t = 1}t2^t + 14\sum^{j - 1}_{t = 1}2^t\right) \ &= \ \nonumber \\
\frac{1}{8}(2^j(j^2 - 4j + 6) - 6 + 9(j2^j - 2^{j + 1} + 2) + 14(2^j - 2)) \ &= \ \nonumber \\
 \frac{1}{8}(j^22^j + 5j2^j + 2^{j + 1} - 16) \ &= \ \nonumber \\
 2^{j - 3}(j^2 + 5j + 2) - 2,
\end{eqnarray}
completing the proof. 
\end{proof}

We can now prove the main result of this section, a closed form for $B_{2k + j, k}$.

\begin{theorem}[Closed Form for $B_{2k + j, k}$]\label{B_{2k+j,k}closedform}
Let $j, k \in \N^+$ with $k > j$. Then the number of ways to split $2k + j$ balls into nonempty bins so the most crowded bin has exactly $k$ balls is
\begin{equation}\label{B_{2k+j,k}closedformEq}
B_{2k + j, k} \ = \ (k + j + 3)2^{k + j - 2} - (3j^2 + 19j + 18)2^{j - 4}.
\end{equation}
\end{theorem}

\begin{proof} This formula follows from directly substituting \eqref{B2k+j,ksumeval1}, \eqref{B2k+j,ksumeval2}, and \eqref{B2k+j,ksumeval3} into \eqref{B_{2k+j,k}sumformbasedonbinpartition}. Notice that the constant terms cancel each other out. 
\end{proof}

While we successfully found a closed form for $B_{2k + j, k}$ with $k > j$, this method is not tractable as $m$ gets larger. Section \ref{GeneralizedBins} serves in part to generalize our summation formulas to the case where the parameters are not within fixed multiple factors of each other.

\section{Generalized solution for balls into bins with restrictions}\label{GeneralizedBins}

The previous section dealt with the formulae for the Balls in Bins problem without a dominant bin for the case when the number of bins was $\ell$, the most crowded bin had exactly $k$ balls and the total number of balls was $2k + j$ with $j < k$. However, one may wonder if it would be possible for the number of balls to be arbitrary and not dependent on other variables (like $k$ in this case). Therefore, we aim to generalize the concept of putting Balls into Bins with restrictions, as defined in Question \ref{gen_ballsbins}. Let us revisit Definition \ref{bins_restrict} and Question \ref{gen_ballsbins}, which were posed in the Introduction (section \ref{intro1}), to obtain a better understanding of the meaning of the problem and the notation which will be used for the rest of the paper. 

\begin{defi}[Generalized Bins restriction problem] \label{new_bins_restrict}
Let $n, \ell, k \in \N^{+}$. The Generalized Bins restriction problem aims to find the number of ways to split n balls into $\ell$ bins such that the maximum number of balls in each bin is at most $k$. 

\end{defi}

\begin{question}[Generalized Balls into Bins with restrictions problem]\label{new_gen_ballsbins}
Let $n, \ell, k \in \N^{+}$. This problem asks how many ways can we split $n$ balls into $\ell$ non-empty bins such that the most crowded bin has exactly $k$ balls?
\end{question}

The notation associated with the quantity defined in the Generalized Bins restriction problem in the Definition \ref{new_bins_restrict} will henceforth be denoted as $R_{n, \ell, k}$. Note that this is different from the problem we want to tackle, in the sense that the maximum number of balls in each bin is at most $k$, not exactly $k$, which was the case investigated in the earlier sections. Also, the notation associated with the quantity defined in the Generalized Balls into Bins with restriction problem in the Question \ref{new_gen_ballsbins} will henceforth be denoted by $M_{n, \ell, k}$. The main result of this section, Theorem \ref{generalizedBallsIntoBinsRestrictionsII}, will be a formula for $R_{n, \ell, k}$. An intermediate step will be to count the number of configurations described by Definition \ref{new_bins_restrict}. 

\subsection{Formula for generalized bins restriction problem: $R_{n, \ell, k}$}\label{formulaR}

\begin{thm} \label{bins_restrict_formula}
Let $n$, $\ell$, $k \in \N^+$. Then
\begin{equation} \label{formulaR_n_ell_k}
R_{n, \ell, k} \ = \ \sum^{\ell}_{t = 0}(-1)^t{{\ell}\choose{t}}{{n - t(k + 1) + \ell - 1}\choose{\ell - 1}}.
\end{equation}
\end{thm}

\begin{proof} 
We will prove this formula with the help of the Principle of Inclusion and Exclusion. First, we note that the total number of filling $n$ balls into $\ell$ bins with no upper restriction on the number of bins is simply
\begin{equation}\label{stars-bars}
    {n + \ell - 1} \choose {\ell - 1}.
\end{equation}
Now, this would include all the configurations that satisfy the restriction condition that each bin contains at most $k$ and all the configurations which do not satisfy this condition. So, the aim of this proof is to first calculate the total number of configurations that do not satisfy the restriction condition and then subtract that from Equation \eqref{stars-bars}.

Let us assume that the bins have been labelled in some order, indexed from $1$ to $\ell$. Let $A_i$ denote the finite set of all configurations such that the $i^{th}$ bin contains more than $k$ balls. Using this definition, note that the total number of configurations that violate the restriction condition is equal to 
\begin{equation}
    \left| \bigcup\limits_{i=1}^{\ell} A_i \right|
\end{equation}
Using the Principle of Inclusion and Exclusion, this is equivalent to the following formula:
%EHG
\begin{equation} \label{PIE_formula}
\begin{aligned}
    \left| \bigcup\limits_{i=1}^{\ell} A_i \right| \ & =\ \sum_{i = 1}^\ell \left| A_i \right|\ -\ \sum_{1 \le i < j \le \ell} \left| A_i \cap A_j \right| \ +\ \sum_{1 \le i < j < k \le \ell} \left| A_i \cap A_j \cap A_k \right| \ -\  \cdots \\  \ +\ (-1)^{\ell - 1} \left| A_1 \cap A_2 \cdots \cap A_\ell \right|.
\end{aligned}
\end{equation}
Now, to find all such configurations, let us find the cardinalities in \eqref{PIE_formula} by breaking into cases. 

\underline{Terms in the First Summation}: This is the number of configurations such that bin number $i$ contains more than $k$ balls. Note that this statement means that bin $i$ contains at least $k+1$ balls. We first put $k+1$ balls into bin $i$. Now, we just need to distribute the remaining $n - k - 1$ balls into $\ell$ bins (Note that bin $i$ is included because we have filled only the "minimum" number of balls in the bin, and it may contain more than $k+1$ balls). Therefore, the total number of ways to place $n - 1(k+1)$ balls into $\ell$ bins is 
\begin{equation} \label{PIE_1}
    {n - 1(k + 1) + \ell - 1} \choose {\ell - 1}.    
\end{equation}
We need to sum this formula over all possible values of $i$, i.e $\sum_{i = 1}^\ell \left| A_i \right|$. This means that we just need to multiply the formula obtained in Equation \eqref{PIE_1} by $\ell$ to sum over all values of $i$. Therefore, we have the following formula: 
\begin{equation}
    \sum_{i = 1}^\ell \left| A_i \right|\ =\ \ell{{n - 1(k + 1) + \ell - 1} \choose {\ell - 1}} \ =\  {{{\ell} \choose {1}} {{n - 1(k + 1) + \ell - 1} \choose {\ell - 1}}}.
\end{equation}
\underline{Terms in the Second Summation}: This is the number of configurations such that the bins numbered $i$ and $j$ where $i < j$ have more than $k$ balls. This means both the bins $i$ and $j$ contain at least $k + 1$ balls. So we again put $k+1$ balls in each bin. Therefore, we have $n - 2(k+1)$ balls remaining to put into $\ell$ bins. Therefore the number of ways to place $n - 2(k + 1)$ balls into $\ell$ bins is: 
\begin{equation} \label{PIE_2}
    {n - 2(k + 1) + \ell - 1} \choose {\ell - 1}    
\end{equation}
We need to sum this formula over all possible values of $i, j$, i.e $\sum_{1 \le i < j \le \ell} \left| A_i \cap A_j \right|$. The number of ways of choosing two such bins out of $\ell$ bins is the simple combinations formula, i.e ${ {\ell} \choose {2} }$. We just need to multiply the formula obtained in Equation \ref{PIE_2} by ${ {\ell} \choose {2} }$. Therefore, we obtain the following formula: 
\begin{equation}
    \sum_{1 \le i < j \le \ell} \left| A_i \cap A_j \right| \ =\ {{{\ell} \choose {2}} {{n - 2(k + 1) + \ell - 1} \choose {\ell - 1}}}.
\end{equation}
We proceed similarly for the other summations given in the formula of Inclusion and Exclusion, and substitute the formulae obtained into Equation \eqref{PIE_formula} which gives us the following formula: 
\begin{equation} \label{PIE_balls}
       \left| \bigcup\limits_{i=1}^{\ell} A_i \right| \ =\  \sum^{\ell}_{t = 1}(-1)^{t - 1}{{\ell}\choose{t}}{{n - t(k + 1) + \ell - 1}\choose{\ell - 1}}.
\end{equation}
Note that this is the number of all configurations that do not satisfy the restriction condition. So, to obtain the number of configurations satisfying the restriction condition, we just subtract Equation \eqref{PIE_balls} from Equation \eqref{stars-bars} to obtain the formula. Therefore, we obtain 
\begin{equation}
    R_{n, \ell, k} \ = \ {{{n + \ell - 1} \choose {\ell - 1}} - \sum_{t = 1}^{\ell} (-1)^{t - 1}{{\ell}\choose{t}}{{n - t(k + 1) + \ell - 1}\choose{\ell - 1}}}. 
\end{equation}
This equation can be further simplified as follows: \begin{equation}
    R_{n, \ell, k} \ = \ \sum^{\ell}_{t = 0}(-1)^t{{\ell}\choose{t}}{{n - t(k + 1) + \ell - 1}\choose{\ell - 1}}.
\end{equation}
\end{proof}
This formula should actually have the restriction on $t$ that $0 \le t \le \frac{n}{k+1}$, because for the cases where $t \ge \frac{n}{k+1}$, the binomial coefficient ${{n - t(k + 1) + \ell - 1}\choose{\ell - 1}}$ would return erroneous values, depending on the definition of a binomial coefficient, which could be defined for negative values too. That said, we want all of these cases to be equated to $0$ because they all violate the definition of our main problem. So, to make things easier, we can define the Binomial Coefficient ${{n}\choose{k}}$ in a new (albeit equivalent) way to use the formula in Equation \ref{formulaR_n_ell_k}. The following definition of binomial coefficients will be used in the remainder of the paper.

% \textcolor{red}{INSERT nCk formula}

\begin{equation}\label{factdef}
{{n} \choose {k}} \ = \ \begin{cases}
\frac{n!}{k!\,(n \ - \ k)!}, \ \ 
n \ge 0\ \ , k\ge 0\ , \ n \ge k  \\
0, \ \  \ \ \ \ \ \ \ \ \ \ \text{otherwise} 
\end{cases}.
\end{equation}

This definition will ease the computation of the binomial coefficients and the cases associated with these in the coming sections. 

\subsection{Formula for $M_{n,\ell,k}$}\label{formulaM}

We approached the problem of Generalized Balls into Bins with restriction in two different ways to obtain different formulae representing the same problem. One of the highlights of our work is that we are establishing new connections between the quantities $R_{n, \ell, k}$ and $M_{n, \ell, k}$ that have not been previously explored in the literature. The first method utilized in this section features the Principle of Inclusion and Exclusion (abbreviated P.I.E.) to derive \eqref{gen_ballsbins}. 

\begin{thm}[Formula for Generalized Balls into Bins with restrictions problem (I)]
Suppose $n$,$\ell$, $k \in \N^+$ such that $\ell + k -1 \leq n \leq \ell k$. Then the following identity holds:
\begin{equation}\label{formulaM_n_ell_k}
 M_{n, \ell, k} \ = \ \sum^{\ell}_{t = 1}\ {(-1)^{t-1}{{\ell}\choose{t}}{R_{n - t(k - 1) - \ell, \ell - t, k - 1}}}. 
 \end{equation}
\end{thm}

\begin{proof} 
We begin the proof by taking cases for the number of bins which are full (i.e how many bins contain exactly $k$ balls) and then distributing the remaining balls into the bins which are not meant to be full with updated restrictions. After that, we use the Principle of Inclusion to construct to sum up all the cases to prove the theorem. So, for case $1$ (where at least one bin is full), the number of ways we can choose $1$ bin out of $\ell$ bins to be full is simply ${{\ell}\choose{1}}$. Now, let us fill the rest of the bins with the balls, while keeping the new restriction in mind. Let's call these variables $\text{new balls}$, $\text{new bins},$ and $\text{new restrict}$. Note that we have already given $k$ balls to a bin, so the number of remaining bins is $\ell -1$. Also note that none of the remaining $\ell - 1$ bins can remain empty, so we 'put' one ball in each of the remaining $\ell -1$ bins. The number of balls remaining is simply $n - k - \ell + 1$. However, note now that as each of the remaining $\ell - 1$ bins contains $1$ ball. In order to apply the formula for Generalized Bins restriction Equation \eqref{bins_restrict_formula} on these remaining bins, we need to reduce the restriction on the number of balls in a bin from $k$ to $k-1$. So, we have the values of the new variables as follows: 
\begin{eqnarray}\label{GeneralizedBinsVariableReassign}
\text{new balls} \  = \ & n - k - \ell + 1 \\
\text{new bins} \ = \ & \ell - 1 \\
\text{new restrict} \ = \ & k - 1.
\end{eqnarray} 
In order to distribute the remaining balls into bins with the given restriction, we use Theorem \eqref{bins_restrict_formula} to identify $R_{\text{new balls}, \text{new bins}, \text{new restrict}}$ as the total number of ways of distribution, given the above conditions. Finally, we multiply this by ${{\ell}\choose{1}}$ and substitute the values of the variables to get the total number of ways such that at least $1$ bin is full is 
\begin{equation}\label{formulaM_n_ell_k_EqA}
{{\ell}\choose{1}}\ R_{n - k - \ell + 1, \ell - 1, k - 1}.
\end{equation}
However, notice that we have over-counted some cases, specifically the cases where at least 2 bins are full. To explain this in more detail, let $(b_1, b_2, b_3, \cdots , b_\ell)$ denote a particular valid configuration of balls in the bins numbered $1, 2, \cdots , \ell$ with the aforementioned conditions satisfied. Let us assume there are two indices $i$ and $j$ with $i \neq j$ such that $b_i = b_j = k$; that is, bins $i$ and $j$ are full with $k$ balls each. In the explanation above, when we chose a bin that was full, we counted this configuration once when we chose $b_i$ as the bin which is full. We also counted this configuration once when we chose $b_j$ as the bin which is full. So, we need to subtract the case where there are at least two bins full from the above sum. The formula for the case where at least two bins are full and which satisfy the original conditions is 
\begin{equation}\label{formulaM_n_ell_k_EqB}
{{\ell}\choose{2}}\ R_{n - 2(k - 1)  - \ell, \ell - 2, k - 1}.
\end{equation}
However, if we subtract this from the case where at least one bin was full, there would be some cases, like the number of bins with at least 3 bins full, which would be under-counted as these cases would appear in both of the above cases. So we need to add this to the original sum. By invoking the Principle of Inclusion and Exclusion, this would go on until the maximum possible number of bins contain $k$ balls. 

Considering the generalized case of putting $n$ balls in $\ell$ bins such that each bin has at most $k$ balls and at least $s$ bins are full, we have the formula 
\begin{equation}\label{general_PIE_equation}
    {{\ell}\choose{s}}{R_{n - s(k - 1) - \ell, \ell - s, k - 1}}.
\end{equation}
Using this formula along with P.I.E. to sum up all the cases gives us the required formula \eqref{formulaM_n_ell_k}. 
\end{proof}

Note that in the above statement, we use the statement "maximal number of bins possible contain $k$ balls", but have taken the summation of cases up to $\ell$. The above proof is valid due to the definition of ${{n}\choose{k}}$, which equates to $0$ for the case when $n \le 0$ or $n < k$. So the above formula is still valid because the rest of the cases do not contribute anything to the sum. 

This approach involving Inclusion-Exclusion has given us a very complicated formula in terms of $R_{n, \ell, k}$, which itself has a complicated formula! Thus it makes sense to find alternate ways to represent the same problem.

\begin{thm}[Formula for Generalized Balls into Bins with restrictions problem (II)]\label{generalizedBallsIntoBinsRestrictionsII} 
Suppose $n, \ell, k \in \N^+$ such that $\ell + k - 1 \leq n \leq \ell k$. Then the following identity for $M_{n, \ell, k}$ holds:
\begin{multline}\label{genBallsBinsRestEq2}
M_{n, \ell, k} \ = \ R_{n - \ell, \ell, k - 1} - R_{n - \ell, \ell, k - 2} \ = \\ \sum^{\ell}_{t = 0} (-1)^t  {{\ell} \choose{t}}{ \left[ {{n - tk - 1}\choose{\ell - 1}} - {{n - t(k - 1) - 1}\choose{\ell - 1}} \right]}.
\end{multline}
\end{thm}

\begin{proof}
We will prove this identity by counting all the configurations that satisfy only some parts of the definition of $M_{n, \ell, k}$, and then subtract all the configurations which do not satisfy the remaining criteria. Specifically, we temporarily relax the condition that a given bin can only have at most $k$ balls. Let us start by remembering that the formula $M_{n, \ell, k}$ holds true only for non-empty bins. We put one ball in each of the $\ell$ bins. Now, the remaining number of balls is $(n - \ell)$. So, we apply equation \ref{bins_restrict_formula} to fill the bins with the rest of the $n-\ell$ balls. However, note that the restriction must now be reduced from $k$ to $k - 1$ because each bin has a ball and we want at most $k$ balls in each bin. So, the number of ways to distribute $(n - \ell)$ balls into $\ell$ bins such that each bin gets at most $k - 1$ balls is simply $R_{n-\ell, \ell, k - 1}$. 

However, note that we have included some unwanted configurations in the above formula. The statement of Definition \ref{gen_ballsbins} also says that the most crowded bin must contain \textit{exactly} $k$ balls. So, there are some cases where the most crowded bin might contain less than $k$ balls, but other than that, it would satisfy all the conditions. In order to eliminate these cases, note that if the most crowded bin does not contain $k$ balls, then it would contain at most $k - 1$ balls. These configurations fall under Definition \ref{bins_restrict}. Hence the number of ways we can distribute $n$ balls into $\ell$ non-empty bins such that each bin contains at most $k - 1$ balls is $R_{n - \ell, \ell, k - 2}$. Subtracting this from $R_{n-\ell, \ell, k - 1}$ gives the desired answer. 
\end{proof}

\begin{remark}\label{genBallsBinsRestEq2Remark} The restriction $\ell + k - 1 \leq n \leq \ell k$ was not explicitly used in the proof of \eqref{genBallsBinsRestEq2}. However, it is important because of the definition of $M_{n, \ell, k}$. Recall that $M_{n, \ell, k}$ counts the number of ways to split $n$ balls into $\ell$ nonempty bins so that the most crowded bin has exactly $k$ balls; the remaining $\ell - 1$ bins will contain at least one ball each. If this were to be possible for some values of $n$, $\ell$, and $k$, then necessarily $\ell + k - 1 \leq n \leq \ell k$. At least one bin must actually be filled to the maximum capacity, giving the lower bound. On the other hand, given $\ell$ bins with maximum capacity $k$, one obtains the greatest possible number of balls by filling all bins with the maximum capacity. This gives the upper bound. Therefore we are justified in using these bounds in the next lemma as well.
\end{remark}

The natural follow-up question after proving \eqref{genBallsBinsRestEq2} is to ask whether there is a closed form for that sum. We believe that the answer to this is negative, because the presence of $tk$ inside the binomial coefficients makes the sum inaccessible by evaluation techniques such as Snake Oil and the WZ Method (see \cite{Wil} for a detailed explanation of these techniques). As a result, we look to derive estimates instead. These estimates will involve exponential functions for two reasons. First, binomial coefficients are intimately related to exponential functions through Stirling's Formula (see for instance \cite{Bona, Fe, Rob, Sta, Sun}). Second, as we will see, such estimates will produce factors with $kt$ as a power, which is more manageable than $kt$ inside a binomial coefficient. 
The statement and proof of this lemma are technical in nature and are comprised mostly of elementary manipulations of sums and inequalities. Therefore we defer the exact statement of the lemma and its proof to Appendix \ref{genBinRestrictEstimates}.

\subsection{Some identities involving $M_{n, \ell, k}$ and $R_{n, \ell, k}$}\label{restrictedBinsProblemIdentities} %will need to change section name to avoid confusion with previous section--ask Vedant to do this

There are various identities which involve the variables $M_{n \ell , k}$ and $R_{n, \ell, k}$. These identities help us understand the nature of the variables. The term $R_{n, \ell, k}$ used in the majority of the paper here is also known as Polynomial Coefficients in various papers like \cite{Fah1} and \cite{Fah2}. These papers provide a more rigorous definition and expression for the term $R_{n, \ell, k}$. The paper \cite{Fah1} has also explored the various combinatorial interpretations of the Polynomial Coefficients in different areas of Mathematics.

There are many identities of involving the Polynomial Coefficients mentioned in these papers, but almost every paper has proved them using generating functions and their algebraic manipulations. Here, we present the proof of four identities using only the elementary principles of combinatorics. %For a more comprehensive look into the identities, both the papers cited above are a good choice. 

\begin{lem}[Identities involving $R_{n, \ell, k}$]
Some of the prominent identities involving the variable $R_{n, \ell, k}$ are highlighted below: 

\begin{enumerate}

 \item{    \begin{equation}\label{lem_1}
        R_{n, \ell, k}  =   R_{\ell k - n, \ell, k} 
    \end{equation}}
    
 \item{    \begin{equation}\label{lem_2}
        R_{n, \ell, m + k} = \sum\limits_{i=0}^{n} {R_{i, \ell, m}\;R_{n - i, \ell, k} }
    \end{equation}}
    
%          \item{     \begin{equation}\label{lem_3}
%          R_{n, \ell, k} - R_{n - 1, \ell, k} = R_{n , \ell -1 , k } - $ %          R_{n - k - 1, \ell - 1 ,k} 
%          \end{equation}}
    
 \item{   \begin{equation}\label{lem_4}
        R_{n, \ell + 1, k} = \sum\limits_{i = 0}^{k} R_{n - i, \ell, k}
    \end{equation}}
    
\item{     \begin{equation}\label{lem_5}
        R_{n+1, \ell + 1, k} - R_{n, \ell + 1, k} = R_{n+1, \ell, k} - R_{n - k, \ell, k}
    \end{equation}}

\end{enumerate}
\end{lem}

\begin{proof}
Let us prove all the identities from \eqref{lem_1} to \eqref{lem_5} in a sequential order. 

\begin{enumerate}
    \item {$R_{n, \ell, k}  =   R_{\ell k - n, \ell, k}$ : We will use the principle of one-to-one correspondence between two sets to show that the cardinalities of both sets are equal, which would lead us to the above identity. We know that each configuration of balls in the bins with conditions expressed with the formula of $R_{n, \ell, k}$ can be uniquely denoted as $(b_1, b_2, b_3, \cdots , b_\ell)$, where each $b_i$ denotes the number of balls in the bin numbered $i$ from a selection of bins numbered from $1$ to $\ell$. The sum of all $b_i$'s equals $n$. Furthermore, for each $1 \leq i \leq \ell$, $b_i$ is a non-negative integer which can have a maximum value of $k$. Let us denote 
    %\begin{equation} \label{config_represent}
    $\{b_i\}^{\ell}_{i = 1}$
    %\end{equation}
    to be the same configuration $(b_1, b_2, b_3, \cdots , b_\ell)$, but expressed in a more compact form. 
    %Each $b_i$ is a non-negative integer for an integer $i$ such that $1 \le i %\le \ell$ and all $b_i$'s are restricted with the condition $b_i \le k$.
    % Let us denote $\{b_i\}^{\ell}_{i = 1}$ to be the same configuration %$(b_1, b_2, b_3, \cdots , b_\ell)$, but expressed in a more compact form. 
    
    %   Now, consider the term $R_{\ell k - n, \ell, k}$, where each 
    %   configuration can be denoted as the set $(c_1, c_2, \cdots, c_\ell)$, 
    %   where each term $c_i$ denotes the number of balls in the bin numbered 
    %   $i$ from a selection of bins numbered from $1$ to $\ell$ and sum of all %   $c_i$'s equals $\ell k - n$.
    
    Now, for each configuration of balls in bins denoted by $\{b_i\}^{\ell}_{i = 1}$, we have another unique configuration of the form $\{k - b_i\}^{\ell}_{i = 1}$. Note that this configuration has each term of the form $k - b_i$, which is always non-negative due to the fact that $b_i \le k$. Also note that the maximum value of $k - b_i$ is at most $k$ due to the fact that $b_i \ge 0$. Now, this configuration can be said to be a collection of bins with $b_i$ balls in the $i$th bin; the $\ell$ bins are indexed such that each bin is non-empty and has at most $k$ balls. Also, the sum of the number of balls in each bin over all the bins is 
    \begin{equation}\label{countingLeftoverBalls}
    (k - b_1) + (k - b_2) + \cdots + (k - b_\ell) \ = \ \ell k - (b_1 + b-2 + \cdots + b_\ell) \ = \ \ell k - n.
    \end{equation}
    On the other hand, the total number of such configurations of this type can be expressed with the formula of $R_{k \ell - n, \ell, k}$. This means that there is a one-to-one correspondence between the two sets of configurations satisfying their respective conditions. This implies that the cardinalities of both the sets are equal, which implies that $R_{n, \ell, k} = R_{k \ell - n, \ell, k}$. 
    }
    
    \item { $R_{n, \ell, m + k} = \sum\limits_{i=0}^{n} {R_{i, \ell, m}\;R_{n - i, \ell, k} }$ :  To prove this, we proceed by counting in two ways. Suppose we have $\ell$ bins and $n$ balls such that each bin has at most $(m+k)$ balls. The total number of such configurations, due to the simple definition of the Generalized Bins restriction problem, equals $R_{n, \ell, m + k}$ from the Formula  \eqref{bins_restrict_formula}. 
    
    Let us count the same problem from a different perspective. Let each bin have a partition which divides the bin into two parts (say $A$ and $B$) such that one component of the partition $A$ can contain a maximum of $m$ balls and the other component $B$ can contain a maximum of $k$ balls. Note that this problem is still equivalent to Question \ref{gen_ballsbins}, because in both  problems, each configuration has bins with at most $m+k$ balls. Enumerating all the possible cases via partitions, let us fill all the partition components $A$ for every bin. Let us suppose we require $i$ balls $(0 \le i \le  n)$ to fill each partition component A of every bin such that no bin's partition component labelled $A$ has more than $m$ balls. The number of such configurations is $R_{i, \ell, m}$. Now, we must fill the remaining $n-i$ balls into all the partition components labelled $B$ of every bin such that no bin has more than $k$ balls. By a similar argument, the number of such configurations equals $R_{n-i, \ell, k}$.
    
    %As each partition's components $A$ and $B$ of any bin can be interchanged with another bin (because the restriction property would still be true),
    
    Now, note that for any two bins $B_i$ and $B_j$ , the $B_i$s that partition component $A$ can be interchanged with the partition component $A$ of bin $B_j$, and both the bins will still satisfy the restrictions on the bins and their partition components. Hence, the total number of ways to fill $\ell$ bins with $n$ balls such that every bin has at most $(m + k)$ balls is the product of the number of ways to fill the two partitions of the bins. That is, this describes the expression $R_{i, \ell, m}R_{n-i, \ell, k}$.
    
    %due to the Multiplication Principle
    
    %  Therefore, the number of ways of filling all bins 
    
    Now, $i$ can range from $0$ to $n$ due to the number of balls available, so we should sum the total number of such configurations over all possible values of $i$. This is equal to $\sum\limits_{i=0}^{n} {R_{i, \ell, m}\;R_{n - i, \ell, k}}$. So, we have just counted the same problem in two ways, hence we have the identity
    \begin{equation}
         R_{n, \ell, k} \ = \ \sum\limits_{i=0}^{n} {R_{i, \ell, m}\;R_{n - i, \ell, k}}.
    \end{equation}
    }
    
    \item { $R_{n, \ell + 1, k} = \sum\limits_{i = 0}^{k} R_{n - i, \ell, k}$ : We will use the concept of counting in two ways to prove this identity.

    We first find a way to represent $R_{n, \ell + 1, k}$. Note that for each bin $b_i$ satisfying the restriction,
    we have that $0 \le b_i \le k$ for any integer $i$ such that $1 \le i \le \ell + 1$. Now, let us fill the first bin first with some number of balls, and then fill the rest of the bins accordingly. Now, the minimum number of balls we can put in the first bin is $b_1 = 0$, and the maximum number of balls we can fill into a bin is $b_1 = k$. Let us find the number of ways to fill the rest of the balls into the remaining bins for each of the cases that lie within the given restrictions. 
    
    \underline{Term in the First Case}  ($b_1 = 0$): This means the number of balls in the first bin is $0$, so in order to fill the remaining $n - 0 = n$ balls into $\ell + 1 - 1 = \ell$ bins such that each bin contains at most $k$ balls, the number of such configurations is $R_{n, \ell, k}$ from the Theorem \ref{bins_restrict_formula}.
    
    \underline{Term in the Second Case}  ($b_1 = 1$): This means the first bin has a single ball; therefore, in order to fill the remaining $n - 1$ balls into $\ell$ bins, such that each bin contains at most $k$ balls, the number of such configurations is $R_{n-1, \ell, k}$.
    }
   
   For the remaining cases, i.e where $b_1$ ranges from $1$ to $k$, we proceed similarly as shown above. We now discuss the last case.
    
    \underline{Term in the $(k\ +\ 1)^{\text{th}}$ Case}  ($b_1 = k$): This means the first bin has $k$ balls; therefore, in order to distribute the remaining $n - k$ balls amongst $\ell$ bins, such that each bin contains at most $k$ balls, the number of such configurations is $R_{n-k, \ell, k}$.
    Now, note that summing each of the cases above would account for every available configuration having $n$ balls and $\ell + 1$ bins where each bin contains at most $k$ balls. However, the formula associated with the total number of such configurations is represented by $R_{n, \ell + 1, k}$. Therefore, 
    \begin{equation}\label{R_compiled_sum}
    R_{n, \ell, k} + R_{n - 1, \ell, k} + \cdots + R_{n - k, \ell, k} = R_{n, \ell + 1, k} \Longrightarrow R_{n, \ell + 1, k} = \sum\limits_{i = 0}^{k} R_{n - i, \ell, k}.
    \end{equation}
    \item{$R_{n+1, \ell + 1, k} - R_{n, \ell + 1, k} = R_{n+1, \ell, k} - R_{n - k, \ell, k}$} : We will leverage formula \eqref{lem_4} to prove this identity. Using that formula, we have the following results: 
    \begin{equation}\label{5_one}
    R_{n+1, \ell+1, k} = \sum\limits_{i = 0}^{k} R_{n + 1 - i, \ell, k}
    \end{equation}
    \begin{equation}\label{5_two}
    R_{n, \ell + 1, k} = \sum\limits_{i = 0}^{k} R_{n - i, \ell, k}.
    \end{equation}
    Subtracting the equation \eqref{5_two} from the equation \eqref{5_one} gives us
    \begin{equation}\label{5_two_one_sub}
    R_{n+1, \ell+1, k} - R_{n, \ell + 1, k} = \sum\limits_{i = 0}^{k} R_{n + 1 - i, \ell, k} - \sum\limits_{i = 0}^{k} R_{n - i, \ell, k},
    \end{equation}
    Further, splitting the first term from the Equation \eqref{5_one} and the last term from the Equation \eqref{5_two} upon further simplification gives us the following: 
     \begin{equation}\label{5_three}R_{n+1, \ell+1, k} - R_{n, \ell + 1, k} = R_{n+1, \ell, k} - R_{n - k, \ell, k} + \sum\limits_{i = 1}^{k} R_{n + 1 - i, \ell, k} - \sum\limits_{i = 0}^{k-1} R_{n - i, \ell, k}. 
     \end{equation}
     Now, notice that we may shift the index of the first sum in equation \eqref{5_three}
     \begin{equation}\label{5_four}
      \sum\limits_{i = 1}^{k} R_{n + 1 - i, \ell, k} = \sum\limits_{i = 0}^{k-1} R_{n - i, \ell, k}
     \end{equation}
     Substituting the equation \eqref{5_four} into the equation \eqref{5_three} gives us the required identity \eqref{lem_5}.
\end{enumerate}
\end{proof}
    
    These identities help develop intuition about the combinatorial properties of the Balls in Bins with restriction formula $R_{n, \ell, k}$. We can further apply these properties to sum up the formula $R_{n, \ell, k}$ while keeping two terms between $n, \ell, k$ as a constant, and the third term as a variable. 
    
\subsection{Sums associated with $M_{n, \ell, k}$}\label{sumsM_n_ell_k}%section will need new name to not be confused with earlier sections, can we combine with previous sections and adjust this in the transitions? Could just do subsections of the same larger section. Maybe ask Vedant to do this.

In Section \ref{intro1}, Question \ref{questionballsinbins} asked the number of ways to split $n$ balls into any number of nonempty ordered bins where the most crowded bin has $k$ balls. Note that in this question, the number of bins was not provided. A natural variant of this question would be to pose the same question when the other variables, such as the number of balls and the restriction on the bin, do not have fixed values. We have already discussed the case with a variable number of bins in sections $2$ and $3$, that was denoted by $B_{n, k}$. Now, we try keeping the other factors as variables and find the appropriate sums associated with these terms.  %We provide the proper definition associated with these questions and the formulae in the following part. 

% The first of these questions could be framed in the following way. In the formula associated with restricted balls and bins with restrictions, which is represented by $M_{n, \ell, k}$, suppose we only know the values of $\ell, k$, then, what could be the total number of configurations possible with no fixed value of $n$, while still following the rules given in the original problem? For this, we mention the following definitions. 

\begin{defi}\label{sum_1}
Let $n, \ell \in \N^{+}$. Then, $K_{n, \ell}$ represents the total number of ways to fill $\ell$ non-empty bins with $n$ balls without any restrictions. 
\end{defi}

Note that this is the same as enumerating all the possible configurations of balls and non-empty bins with the restriction on each bin ranging from $0$ to $n - \ell + 1$. Note that technically a restriction of $0$ is not possible unless the number of balls is $0$, but we have accounted for all such cases in the formula of $R_{n, \ell, k}$. Therefore we can write this sum as 
\begin{equation}\label{K_partition_M}
     K_{n, \ell} \ = \ \sum_{i = 1} ^{n - \ell + 1} M_{n, \ell, i}.
\end{equation}
Now, if we represent $M_{n, \ell, k}$ in terms of $R_{n, \ell, k}$, it can be seen that the sum is actually a telescoping sum which cancels all of the terms and leaves only $R_{n - \ell, \ell, n - \ell + 1}$.

\begin{thm}\label{sum_1_thm}
Let $k, \ell \in \N^+$. Then
\begin{equation} \label{KFormula}
K_{n, \ell}\ = \ {{n - 1} \choose {\ell - 1}}.
\end{equation}
\end{thm}

\begin{proof}
Notice that there are no restrictions on how many balls can be in each bin. Also, Definition \ref{sum_1} simply represents the number of ways to fill $n$ balls into $\ell$ non-empty bins, which is already a very well-known problem in combinatorics. To give a perspective for the formula associated with $K_{n, \ell, k}$, we start by filling each bin with one ball. This is because we do not have any empty bins. So, we have $n-\ell$ balls remaining to be distributed into $\ell$ bins, such that each bin gets $0$ or more balls. The number of ways this can be done is $${{(n - \ell) + \ell - 1} \choose{\ell - 1}} =  {{n - 1} \choose {\ell - 1}}$$ from the Stars and Bars argument. 

% Balls into Bins formula.

%\textcolor{red}{(INCLUDE THE GENERAL DEFINITION OF BALLS INTO BINS FORMULA)}.

\end{proof}

The next definition focuses on having a constant number of bins, and a restriction on the maximum number of balls that can be present in the bins. The main task is to enumerate all possibilities of filling these bins with balls while satisfying the given restrictions. 

\begin{defi}\label{sum_2}
Let $k, \ell \in \N^{+}$. Then $N_{\ell, k}$ represents the total number of ways to fill any number of balls into $\ell$ non-empty bins such that the most crowded bin/bins contain exactly $k$ balls. 
\end{defi}

In this case, we are varying the number of balls, whereas the bins and the restriction on the bin remain the same. The minimum number of balls required in the $\ell$ bins such that all the bins are non-empty and atleast $1$ bin contains $k$ balls would be equal to $k + \ell - 1$. The maximum number of balls would be utilized when all bins have $k$ balls, which would be equal to $k\ell$. Therefore, we can write the sum of all such configurations as

\begin{equation}\label{N_partition_M}
     N_{\ell, k} \ = \ \sum_{i\ =\ k + \ell - 1} ^{k\ell} M_{i, \ell, k}.
\end{equation}

\begin{thm}\label{sum_2_thm}
Let $k, \ell \in \N^+$. Then
\begin{equation}\label{N_ell_k_formula}
N_{\ell, k}\ = \ k^\ell\ -\ (k-1)^\ell.
\end{equation}
\end{thm}

\begin{proof} We will solve this problem by first counting a more general version of the problem, and then subtracting all the cases violating the original conditions. Note that we have these restrictions: 
\begin{enumerate}
    \item The most crowded bin/(s) have exactly $k$ balls.
    \item All bins are non-empty.
\end{enumerate} 
Let us partially fulfill these conditions, and fill all the bins such that all bins are non-empty and the maximum number of balls in the bins are at most $k$. We have $\ell$ bins, and each bin could have anywhere between $1$ to $k$ balls; as there is no restriction on the total number of balls available, the total number of filling $\ell$ bins with this condition is $k^{\ell}$. 
Note that in the above case, we have also counted the configurations which has the most crowded bin with less than $k$ balls. For example, consider the configuration $(k-1, k-1, \cdots, k-1)$ where all $\ell$ bins contain $(k-1)$ balls. Here, the most crowded bins contain $(k-1)$ balls. However, in the original definition, the most crowded bin has exactly $k$ balls. Thus in order to count only those configurations which have the most crowded bin with exactly $k$ balls, we subtract all such configurations whose most crowded bin contains less than $k$ balls from $k^{\ell}$ as shown above. 
The number of ways to fill $\ell$ bins such that the most crowded bin(s) contain at most $(k-1)$ balls is simply $(k-1)^{\ell}$. Subtracting the two desired expressions yields the  result.
\end{proof}

% \begin{defi}\label{sum_3}
% Let $n, k \in \N^{+}$. Then, $B_{n, k}$ represents the total number of ways to fill $n$ balls into any number of non-empty bins with balls such that the most crowded bin/bins contain exactly $k$ balls. 
% \end{defi}
%can comment out the above definition if we do not add a theorem

\section{Interpretation in counting solutions to integer equations}\label{intEqn}

The problem studied in depth through earlier sections of this paper happens to be analogous with a class of problems that have a more robust literature. One elementary problem that frequents itself in mathematics textbooks such as \cite{Bona, Mil, Wil} is to find the number of positive integer solutions to
\begin{equation}\label{intSolutionCanonical}
x_1 + x_2 + \dots + x_{\ell} \ = \ n,
\end{equation}
where $\ell \leq n$. This equation with various choices for the coefficients is closely studied in additive and enumerative combinatorics papers such as \cite{Buk, Mur, Ruz93, Ruz95}. Incidentally, this problem also motivated the development of the ``stars-and-bars" argument we used numerous times throughout the paper. There are many restrictions that can be added to these solutions. One of particular relevance is where we require that the maximum attained value of $\{x_1, x_2, \dots, x_n\}$ is some $k > 0$; this is relevant because finding a formula for $M_{n, \ell, k}$ is clearly equivalent to finding the number of positive integer solutions to \eqref{intSolutionCanonical} with this restriction, and the author is not aware of previously published resources that provide a closed form to the solution of this problem. Among other things, this paper gives us such a closed form when $n \leq 2k$ or when $\ex j \in \N^+$ such that $j < k$ and $n = 2k + j$. Here are two examples where we demonstrate that our formulas agree with the list of solutions to \eqref{intSolutionCanonical} that is generated by inspection.
\begin{example} Using  \eqref{B_{2k, k}FixedBin}, we calculate
\begin{equation}\label{intSolutionEx1A}
M_{8, 5, 4} \ = \ 5{{4 - 1}\choose{5 - 2}} \ = \ 5.
\end{equation}
On the other hand, the list of solutions to \eqref{intSolutionCanonical} when $n = 8, \ell = 5, k = 4$ is
\begin{eqnarray}\label{intSolutionEx1B}
(x_1, x_2, x_3, x_4, x_5) \in \{(4, 1, 1, 1, 1), (1, 4, 1, 1, 1), (1, 1, 4, 1, 1), (1, 1, 1, 4, 1), \\ (1, 1, 1, 1, 4)\nonumber \},
\end{eqnarray}
and there are $5$ total solutions on this list.
\end{example}
\begin{example} Using the third case in \eqref{totalFixedBin}, we calculate
\begin{equation}\label{intSolutionEx2A}
M_{8, 4, 3} \ = \ 4{{3 + 2 - 1}\choose{4 - 2}} - \frac{4^2 - 4}{2}{{2 - 1}\choose{4 - 3}} \ = \ 18.
\end{equation}
On the other hand, the list of solutions to \eqref{intSolutionCanonical} when $n = 8, \ell = 4, k = 3$ is
\begin{eqnarray}\label{intSolutionEx2B}
(x_1, x_2, x_3, x_4) \in \{(3, 3, 1, 1), (3, 1, 3, 1), (3, 1, 1, 3), (1, 3, 3, 1), (1, 3, 1, 3), \\ 
(1, 1, 3, 3), (3, 2, 2, 1), (3, 2, 1, 2), (3, 1, 2, 2), (2, 3, 2, 1),\nonumber \\ 
(2, 3, 1, 2), (1, 3, 2, 2), (2, 2, 3, 1), (2, 1, 3, 2), (1, 2, 3, 2), \nonumber \\
(2, 2, 1, 3), (2, 1, 2, 3), (1, 2, 2, 3) \} \nonumber,
\end{eqnarray}
and there are $18$ total solutions on this list.
\end{example}

\section{Conclusion and future work}\label{conclusion}

We studied two variants of the same problem in this paper. We counted how many ways to split $n$ balls into nonempty, ordered bins so that the most crowded bin has exactly $k$ balls, for certain values of $n$ and $k$. In the first variant, we fixed the total number of bins that were allowed; in the second variant, we permitted any number of bins to be used, as long as all the bins were nonempty. Clearly, the solution to the second problem follows very quickly from the first, but the first problem is of interest in its own right. As we demonstrated in Section \ref{intEqn}, the first problem yields a previously unknown closed form for counting the number of solutions to certain integer equations.

That being said, we have yet to find a closed form for $B_{n, k}$ for all $n, k \in \N^+$. The case where $\ex j < k$ such that $n = 2k + j$ may be most illustrative to how we handle the case where $n = mk + j$ for some $m \geq 3$. Thus, the natural open problem to consider is whether the technique used in Section \ref{B_{2k+j,k}closedFormSection} can be extended to this more general case. Finding a closed form for this general case will completely solve the problem at hand, because whenever $n > k$, there exists $m \in \N^+, 0 \leq j < k$ such that $n = mk + j$. The case where $m = 1$ is the dominant bin case considered in Section \ref{binDomCluster}, and the $m = 2$ case is handled in Section \ref{ballsinbinsnondom}. 

In a previous discussion on public forums, \footnote{ \url{https://math.stackexchange.com/questions/3548108}}, 
people have approximated the distribution of $M_{n, \ell, k}$ to a normal distribution. Here the parameter is $\ell$ (i.e the number of bins); thus it may be worthwhile to solidify this theory.%

There are also other aspects of this problem which could be explored in the future. One of them is exploring the behaviour of the average number of bins, given the number of balls $n$, the restriction on each bin $k$, and the fact that at least 1 bin contains exactly $k$ balls. Basically, we are exploring the properties of $M_{n, \ell, k}$ with $\ell$ as a variable. The average number of bins required in this case would just be the weighted sum of the number of bins required and the number of such configurations divided by the total number of configurations possible for all cases of the number of bins. Using Python libraries such as NumPy and Matplotlib, we explored the nature of the average number of bins with respect to other factors (like number of balls, and the restriction on the bins). Figure \ref{fig:avg_balls} shows the number of configurations for a constant restriction on the bins and an increasing number of balls. Using many simulations of such test cases, it can be safely conjectured that the average number of bins increases if we keep the restriction on bins as a constant value and increase the number of balls. 
 
\begin{figure}%
    \centering
    \subfloat[\centering $n = 10\ ,\ k = 3$]{{\includegraphics[width=5.5cm]{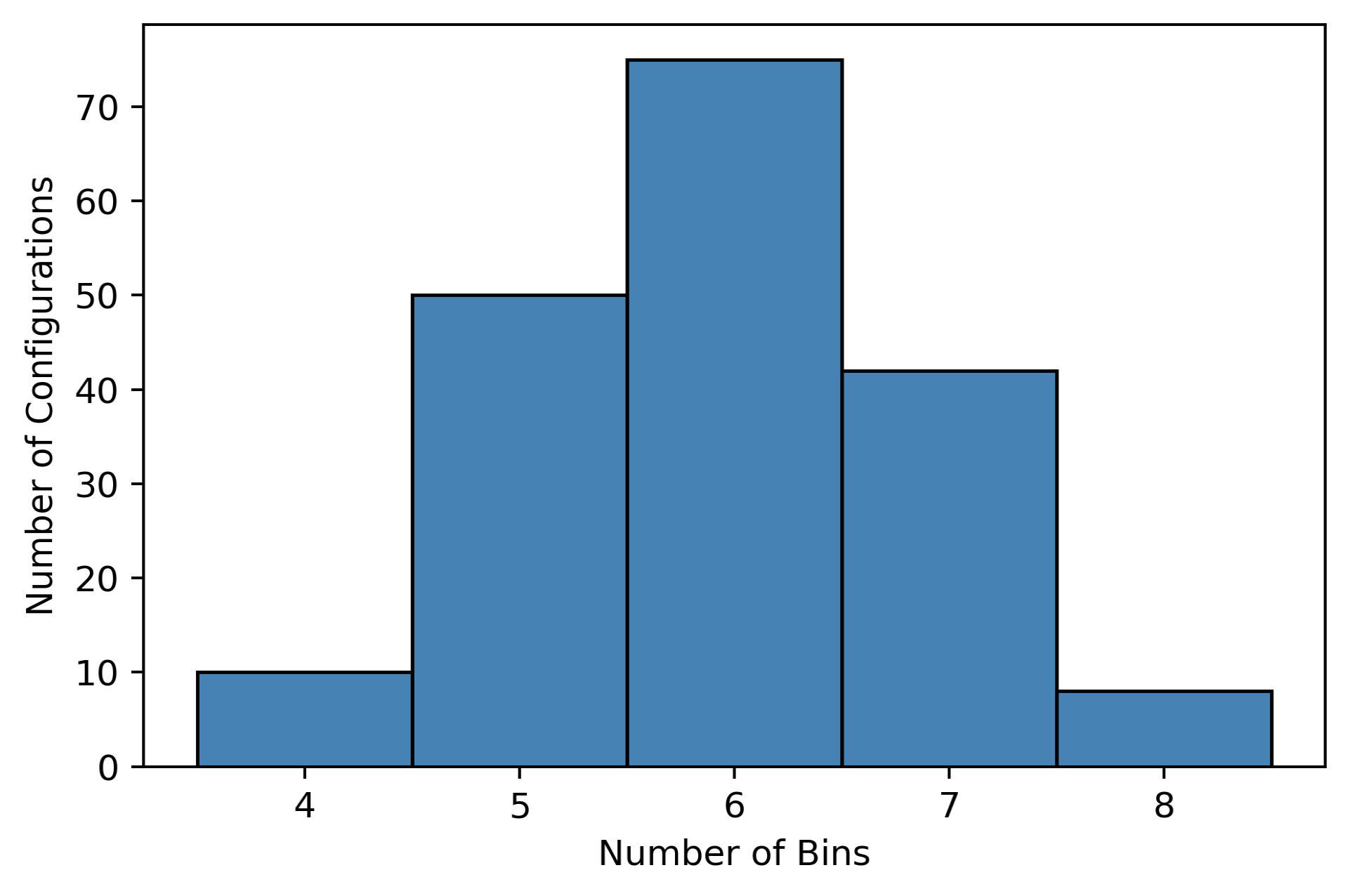} }}%
    \qquad
    \subfloat[\centering $n = 15\ ,\ k = 3$]{{\includegraphics[width=5.5cm]{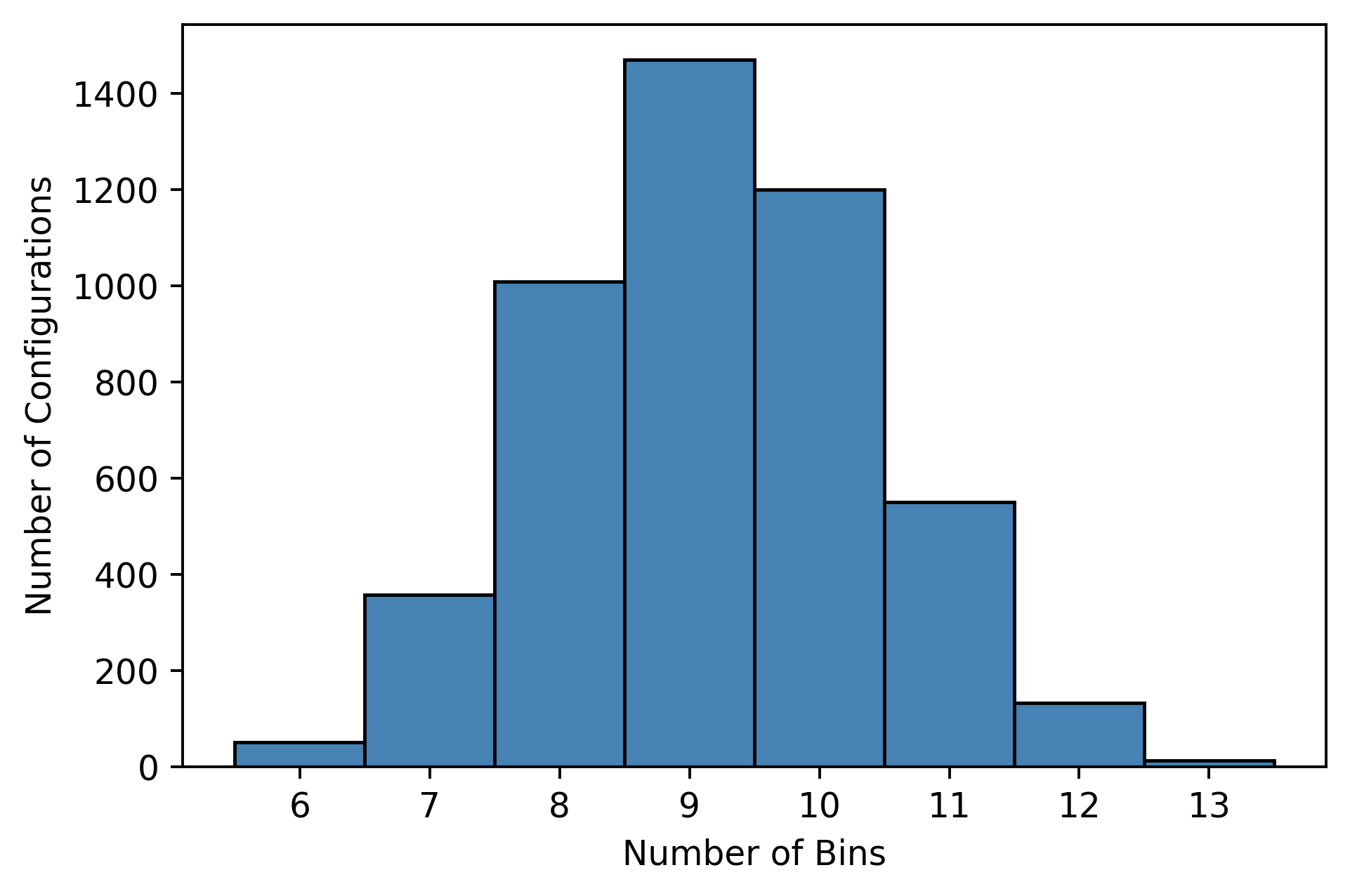} }}%
    \caption{Distribution of the number of configurations with varying number of balls; \\ Our GitHub repository with the code is located at \url{https://github.com/vedantbonde/Balls-in-Bins-Analysis}}%
    \label{fig:avg_balls}%
\end{figure}

Finally, there are some other problems to consider for future research. One variant of the problems studied in this paper is as follows: we can enumerate ways to split $n$ balls into nonempty, ordered bins so that the most crowded bin has exactly $k$ balls, and exactly $t$ bins have this many balls. This is a sensible problem to study because the calculations in Section \ref{B_{2k+j,k}closedFormSection} implicitly broke the problem at hand into the cases $t = 1$ and $t = 2$. There is also an asymptotic problem to consider. For fixed $n$ and $k$, we can consider the list of all configurations where we split $n$ balls into nonempty, ordered bins so the most crowded bin has exactly $k$ balls as a probability space with the parameter being the number of bins used. Then we ask what probability distribution is resembled as $n \rightarrow \infty$ for a fixed $k$. 

\vspace{0.2 cm}

\noindent {\bf Acknowledgements.} We thank our colleague Steven J. Miller and the Polymath Jr. Program for facilitating this collaboration.

\appendix

\section{Some useful binomial coefficient identities}\label{BCID}

In this appendix we will provide proofs of some binomial coefficient identities that are used throughout the paper and have many uses outside the context of this paper. The proofs in this section are very similar in nature to those in the Appendix of \cite{Che}. They may be regarded as well-known but we provide proofs for sake of completeness and for future reference.
\begin{lemma} For any $n \in \N^+,$
\begin{equation}\label{BCIdentitym=1}
\sum^{n}_{k = 0}k{{n}\choose{k}} \ = \ n2^{n - 1}.
\end{equation}
\end{lemma}
\begin{proof}
This follows immediately from using the formula for the mean of a binomial random variable on page 330 of \cite{Mil}, using value of the parameter $p = \frac{1}{2}$.
\end{proof}
\begin{lemma} For any $n \in \N^+$,
\begin{equation}\label{BCIdentitym=2}
\sum^{n}_{k = 0}k^2{{n}\choose{k}} \ = \ n(n + 1)2^{n - 2}.
\end{equation}
\end{lemma}
\begin{proof} We aim to rewrite the left-hand side of \eqref{BCIdentitym=2} so that we can invoke the formula \eqref{BCIdentitym=1}. To do this we turn the binomial coefficient into ${{n}\choose{k - 1}}$:
\begin{equation}\label{BCIdentitym=2Step1}
\sum^{n}_{k = 0}k^2{{n}\choose{k}} \ = \ \sum^{n}_{k = 0}k^2\frac{n!}{k!(n - k)!} \ = \ \sum^{n}_{k = 0}k(n - k + 1){{n}\choose{k - 1}}.
\end{equation}
We notice this sum's first term vanishes and shift the index of it:
\begin{equation}\label{BCIdentitym=2Step2}
\sum^{n}_{k = 0}k(n - k + 1){{n}\choose{k - 1}} \ = \ \sum^{n}_{k = 1}k(n - k + 1){{n}\choose{k - 1}} \ = \ \sum^{n - 1}_{k = 0}(k + 1)(n - k){{n}\choose{k}}.
\end{equation}
Now we decompose this into three sums:
\begin{equation}\label{BCIdentitym=2Step3}
\sum^{n - 1}_{k = 0}(k + 1)(n - k){{n}\choose{k}} \ = \ -\sum^{n - 1}_{k = 0}k^2{{n}\choose{k}} + (n - 1)\sum^{n - 1}_{k = 0}k{{n}\choose{k}} + n\sum^{n - 1}_{k = 0}{{n}\choose{k}}.
\end{equation}
All three sums look familiar if we include the $k = n$ terms: the first sum is the opposite of the left-hand side of \eqref{BCIdentitym=2}, the second sum is the left-hand side of \eqref{BCIdentitym=1}, and the right-hand side is the binomial expansion of $(1 + 1)^n$. Thus we decide to add the $k = n$ terms and then subtract them:
\begin{multline}\label{BCIdentitym=2Step4}
 -\sum^{n - 1}_{k = 0}k^2{{n}\choose{k}} + (n - 1)\sum^{n - 1}_{k = 0}k{{n}\choose{k}} + n\sum^{n - 1}_{k = 0}{{n}\choose{k}} \ = \ \\
 -\sum^{n}_{k = 0}k^2{{n}\choose{k}} + (n - 1)\sum^{n}_{k = 0}k{{n}\choose{k}} + n\sum^{n}_{k = 0}{{n}\choose{k}} + n^2{{n}\choose{n}} - (n - 1)n{{n}\choose{n}} + n{{n}\choose{n}}.
\end{multline}
It turns out the last three terms on the right-hand side of \eqref{BCIdentitym=2Step4} cancel each other out. In addition, everything in \eqref{BCIdentitym=2Step4} is equal to the left-hand side of \eqref{BCIdentitym=2}, so we can treat those quantities like an equation and rearrange it to obtain
\begin{equation}\label{BCIdentitym=2Step5}
2\sum^{n}_{k = 0}k^2{{n}\choose{k}} = (n - 1)\sum^{n}_{k = 0}k{{n}\choose{k}} + n\sum^{n}_{k = 0}{{n}\choose{k}}.
\end{equation}
The desired result follows from a direct application of \eqref{BCIdentitym=1} and the fact that the rightmost sum in \eqref{BCIdentitym=2Step5} is equal to $2^n$.
\end{proof}

\begin{remark} This process can be repeated recursively to evaluate sums of the form $\sum^{n}_{k = 0}k^m{{n}\choose{k}}$, but in this paper we only need the results for $m = 1$ and $m = 2$. There is no known closed form for this sum for general values of $m$. See \cite{Boy1, Boy2} for more information on past explorations with this family of sums, including a discussion of how Stirling's triangle was developed. \end{remark}

\section{Estimates on enumerating generalized balls into bins with restrictions}\label{genBinRestrictEstimates}

Before obtaining our estimate on Equation \eqref{genBallsBinsRestEq2}, we need one other preliminary lemma. This identity may be regarded as well-known but we provide a quick proof for sake of completeness.

\begin{lemma}\label{genBallsBinsRestLem1} Let $m \in \N^+$. Then
\begin{equation} \label{genBallsBinsRestLem1Eq1}
\sum_{\substack{0 \leq t \leq m \\ t \ \text{even}}}{{m}\choose{t}} \ = \  \sum_{\substack{0 \leq t \leq m \\ t \ \text{odd}}}{{m}\choose{t}} \ = \  2^{m - 1}.   
\end{equation}

\end{lemma}

\begin{proof} The following two identities follow from applying the Binomial Theorem to expand $(1 + 1)^{m}$ and $(1 - 1)^{m}$, respectively.

\begin{equation} \label{genBallsBinsRestLem1Eq1A}
\sum^{m}_{t = 0}{{m}\choose{t}} \ = \ 2^{m}   
\end{equation}

\begin{equation} \label{genBallsBinsRestLem1Eq1B}
\sum^{m}_{t = 0}(-1)^t{{m}\choose{t}} \ = \ 0.   
\end{equation}

We can add \eqref{genBallsBinsRestLem1Eq1A} and \eqref{genBallsBinsRestLem1Eq1B} together. In doing so, the terms with odd index $t$ cancel out while the terms with even index $t$ double in magnitude. This proves the first equality in \eqref{genBallsBinsRestLem1Eq1}. Subtracting this equality from \eqref{genBallsBinsRestLem1Eq1A} yields the second equality in \eqref{genBallsBinsRestLem1Eq1}, completing the proof.

\end{proof}

\begin{lemma}\label{genBallsBinsRestLem2}
Let $n, \ell, k \in \N^+$ such that $k \leq n \leq \ell k$ and $n \geq 2$. Then we have the following estimates on $M_{n, \ell, k}$:
\begin{multline} \label{genBallsBinsRestLem2Eq2-1}
M_{n, \ell, k} \ \leq \ 2{{\ell}\choose{\al}}{{n - \al k - 1}\choose{\ell - 1}} + 2{{\ell}\choose{\be}}{{n - \be(k - 1) - 1}\choose{\ell - 1}} + \\ %%%%%
\left(\frac{2^{\ell - 1}}{(\ell - 1)!}\cdot \frac{(n - 1)^{n - \frac{1}{2}}e^{1 - \ell + \ell k}e^{\frac{1}{12(n - \ell k - 1)} - \frac{1}{12(n - \ell) + 1}}}{(n - (\al - 1) k - \ell)^{n - \ell k - \ell + \frac{1}{2}}}\right) - \\ %%%%%%
\left(\frac{2^{\ell}}{(\ell - 1)!} \cdot (n - \ell)^{-\ell - 1}e^{1 - \ell}e^{\frac{1}{12n - 11} - \frac{1}{12}}\right) + \\%%%%%% 
\left(\frac{2^{\ell - 1}}{(\ell - 1)!}\cdot \frac{(n - 1)^{n - \frac{1}{2}}e^{1 - \ell + \ell k}e^{\frac{1}{12(n - \ell (k - 1) - 1)} - \frac{1}{12(n - \ell) + 1}}}{(n - (\be - 1)(k - 1) - \ell)^{n - \ell(k - 1) - \ell + \frac{1}{2}}}\right).
\end{multline} 

\begin{multline} \label{genBallsBinsRestLem2Eq2-2}
M_{n, \ell, k} \ \geq \  -2{{\ell}\choose{\al}}{{n - \al k - 1}\choose{\ell - 1}} - 2{{\ell}\choose{\be}}{{n - \be(k - 1) - 1}\choose{\ell - 1}} + \\ %%%%%%%%
\left(\frac{2^{\ell}}{(\ell - 1)!} \cdot (n - \ell)^{-\ell - 1}e^{1 - \ell}e^{\frac{1}{12n - 11} - \frac{1}{12}}\right) - \\ %%%%%%
\left(\frac{2^{\ell - 1}}{(\ell - 1)!}\cdot \frac{(n - 1)^{n - \frac{1}{2}}e^{1 - \ell + \ell k}e^{\frac{1}{12(n - \ell k - 1)} - \frac{1}{12(n - \ell) + 1}}}{(n - (\al - 1) k - \ell)^{n - \ell k - \ell + \frac{1}{2}}}\right) - \\ %%%%%%%
\left(\frac{2^{\ell - 1}}{(\ell - 1)!}\cdot \frac{(n - 1)^{n - \frac{1}{2}}e^{1 - \ell + \ell k}e^{\frac{1}{12(n - \ell (k - 1) - 1)} - \frac{1}{12(n - \ell) + 1}}}{(n - (\be - 1)(k - 1) - \ell)^{n - \ell(k - 1) - \ell + \frac{1}{2}}}\right).
\end{multline}
The constants $\al$ and $\be$ are defined as follows:
\begin{equation} \label{genBallsBinsRestLem2ConstDefEq1}
\al \ := \ \max\{t \in \{0, 1, \dots, \ell\}, n - tk - 1 \geq \ell - 1\}
\end{equation}
\begin{equation} \label{genBallsBinsRestLem2ConstDefEq2}
\be \ := \ \max\{t \in \{0, 1, \dots, \ell\}, n - t(k - 1) - 1 \geq \ell - 1\}.       
\end{equation}
\end{lemma}

\begin{proof}
The idea behind the proof is as follows: recall that Equation \eqref{genBallsBinsRestEq2}'s summation form is a sum over $t$ ranging from $0$ to $\ell$. Each term in the sum has binomial coefficients ${{n - tk - 1}\choose{\ell - 1}}$ and ${{n - t(k - 1) - 1}\choose{\ell - 1}}$. We will estimate these binomial coefficients from above and below, and then use the alternating nature of the sum \eqref{genBallsBinsRestEq2} to bound each term of the sum appropriately, based on the parity of $t$. An application of \eqref{genBallsBinsRestLem1Eq1} to the resulting bounds will complete the proof.

We will also need to rule out some trivial cases for how $n$, $\ell$, and $k$ relate in order to validate the forthcoming calculations. Recall that $M_{n, \ell, k}$ denotes the number of ways to split $n$ balls into $\ell$ nonempty bins where the most crowded bin has exactly $k$ balls. In this setup, the total number of balls is at most $\ell k$, where all $\ell$ bins are filled to maximum capacity. That is, $n 
\leq \ell k$. On the other hand, at least one bin must be filled to maximum capacity, so $k \leq n$. See Remark \ref{genBallsBinsRestEq2Remark}. 

If $n - k\ell - \ell < 0$ then $n - k\ell - 1 < \ell - 1$, which means the binomial coefficient ${{n - tk - 1}\choose{\ell - 1}}$ vanishes when $t = \ell$. The binomial coefficient estimates we will prove are
\begin{equation} \label{genBallsBinsRestLem2Eq2A}
{{n - tk - 1}\choose{\ell - 1}} \leq \frac{1}{(\ell - 1)!}\cdot \frac{(n - 1)^{n - \frac{1}{2}}}{(n - (\al - 1) k - \ell)^{n - \ell k - \ell + \frac{1}{2}}}e^{1 - \ell + \ell k}e^{\frac{1}{12(n - \ell k - 1)} - \frac{1}{12(n - \ell) + 1}}
\end{equation}
\begin{equation} \label{genBallsBinsRestLem2Eq2B}
 {{n - tk - 1}\choose{\ell - 1}} \ \geq \ \frac{1}{(\ell - 1)!} \cdot (n - \ell)^{-\ell - 1}e^{1 - \ell}e^{\frac{1}{12n - 11} - \frac{1}{12}}
\end{equation}
\begin{multline} \label{genBallsBinsRestLem2Eq2C}
 {{n - t(k - 1) - 1}\choose{\ell - 1}} \ \leq \ \frac{1}{(\ell - 1)!}\cdot \frac{(n - 1)^{n - \frac{1}{2}}}{(n - (\be - 1)(k - 1) - \ell)^{n - \ell(k - 1) - \ell + \frac{1}{2}}} \\
 e^{1 - \ell + \ell k}e^{\frac{1}{12(n - \ell (k - 1) - 1)} - \frac{1}{12(n - \ell) + 1}}
\end{multline}
\begin{equation} \label{genBallsBinsRestLem2Eq2D}
 {{n - t(k - 1) - 1}\choose{\ell - 1}} \ \geq \ \frac{1}{(\ell - 1)!} \cdot (n - \ell)^{-\ell - 1}e^{1 - \ell}e^{\frac{1}{12n - 11} - \frac{1}{12}},
\end{equation}
where $0 \leq t < \al$ in \eqref{genBallsBinsRestLem2Eq2A} and \eqref{genBallsBinsRestLem2Eq2B}, and $0 \leq t < \be$ in \eqref{genBallsBinsRestLem2Eq2C} and \eqref{genBallsBinsRestLem2Eq2D}. Notice that if $t > \al$, then the binomial coefficient ${{n - tk - 1}\choose{\ell - 1}}$ vanishes; if $t > \be$ then the binomial coefficient ${{n - t(k - 1) - 1}\choose{\ell - 1}}$ vanishes. With that in mind, \eqref{genBallsBinsRestEq2} is rewritten by truncating the sum:
\begin{multline} \label{genBallsBinsRestLem2Eq2C}
M_{n, \ell, k} \ = \ (-1)^{\al}{{\ell}\choose{\al}}{{n - \al k - 1}\choose{\ell - 1}} - (-1)^{\be}{{\ell}\choose{\be}}{{n - \be(k - 1) - 1}\choose{\ell - 1}} + \\
\sum^{\al - 1}_{t = 0} (-1)^t{{\ell} \choose{t}}{{n - tk - 1}\choose{\ell - 1}} - \sum^{\be}_{t = 0}(-1)^t{{\ell}\choose{t}}{{n - t(k - 1) - 1}\choose{\ell - 1}}.
\end{multline}
The central Stirling-type estimates we will invoke are
\begin{equation} \label{StirlingEstimateRobFe}
\sqrt{2\pi}m^{m + \frac{1}{2}}e^{-m}e^{\frac{1}{12m + 1}} \leq m! \leq \sqrt{2\pi}m^{m + \frac{1}{2}}e^{-m}e^{\frac{1}{12m}}\end{equation}
These are proven in \cite{Fe, Rob}. There are also similar estimates proven in \cite{Sun}, but those are only relevant when $n$ is large. To obtain the upper bound in \eqref{genBallsBinsRestLem2Eq2A}, use the definition of the binomial coefficient along with both estimates in \eqref{StirlingEstimateRobFe}:
\begin{multline}\label{genBallsBinsRestLem2Eq2D}
{{n - tk - 1}\choose{\ell - 1}} \ = \ \frac{(n - tk - 1)!}{(\ell - 1)!(n - tk - \ell)!} \ \\ \leq \ \frac{1}{(\ell - 1)!} \cdot \frac{(n - tk - 1)^{n - tk - \frac{1}{2}}e^{-n + tk + 1}e^{\frac{1}{12(n - tk - 1)}}}{(n - tk - \ell)^{n - tk - \ell + \frac{1}{2}}e^{-n + tk + \ell}e^{\frac{1}{12(n - tk - \ell) + 1}}}.
\end{multline}
However, the above is only guaranteed to hold if $t < \al$. It can be seen that for any $t < \al$, we have $n - tk - 1 > \ell - 1$. However, if $n - \al k - 1 = \ell - 1$, then the binomial coefficient in question is actually equal to $1$. Rather than break into cases depending on whether $\al$ has this property or not, we simply separate the term where $t = \al$ within our summation formula \eqref{genBallsBinsRestLem2Eq2C}. We use analogous reasoning to justify separating the $t = \be$ term from the second sum in \eqref{genBallsBinsRestLem2Eq2C}. Henceforth in our further estimation of \eqref{genBallsBinsRestLem2Eq2D}, we will have $0 \leq t \leq \al - 1$. As a consequence, $n - tk - \ell > 0$, and we will no longer need to be concerned with inadvertent division by $0$. The goal is to estimate \eqref{genBallsBinsRestLem2Eq2D} as tightly as possible while removing all dependencies on $t$. Our next step in this vein will be to estimate $n - tk - 1 \leq n - 1$ in the numerator and $n - tk - \ell \geq n - (\al - 1) k - \ell$ in the denominator:
\begin{equation}\label{genBallsBinsRestLem2Eq2E}
{{n - tk - 1}\choose{\ell - 1}} \ \leq \ \frac{1}{(\ell - 1)!}\cdot \frac{(n - 1)^{n - tk - \frac{1}{2}}e^{-n + tk + 1}e^{\frac{1}{12(n - tk - 1)}}}{(n - (\al - 1) k - \ell)^{n - tk - \ell + \frac{1}{2}}e^{-n + tk + \ell}e^{\frac{1}{12(n - tk - \ell) + 1}}}.
\end{equation}
Now we estimate exponents in the numerator from above, and exponents in the denominator from below, to obtain
\begin{equation}\label{genBallsBinsRestLem2Eq2F}
{{n - tk - 1}\choose{\ell - 1}} \ \leq \ \frac{1}{(\ell - 1)!}\cdot \frac{(n - 1)^{n - \frac{1}{2}}e^{-n + \ell k + 1}e^{\frac{1}{12(n - \ell k - 1)}}}{(n - (\al - 1) k - \ell)^{n - \ell k - \ell + \frac{1}{2}}e^{-n + \ell}e^{\frac{1}{12(n - \ell) + 1}}}.
\end{equation}
Notice that $n - 1 > 0$ since we assumed $n \geq 2$. Finally, this inequality can be consolidated slightly to obtain \eqref{genBallsBinsRestLem2Eq2A}. This concludes the proof of the upper bound for ${{n - tk - 1}\choose{\ell - 1}}$. Now we will prove the lower bound \eqref{genBallsBinsRestLem2Eq2B} in a similar manner. Use the definition of the binomial coefficient along with both estimates in \eqref{StirlingEstimateRobFe}:
\begin{multline}\label{genBallsBinsRestLem2Eq2G}
{{n - tk - 1}\choose{\ell - 1}} \ = \ \frac{(n - tk - 1)!}{(\ell - 1)!(n - tk - \ell)!} \ \\ \geq \ \frac{1}{(\ell - 1)!} \cdot \frac{(n - tk - 1)^{n - tk - \frac{1}{2}}e^{-n + tk + 1}e^{\frac{1}{12(n - tk - 1) + 1}}}{(n - tk - \ell)^{n - tk - \ell + \frac{1}{2}}e^{-n + tk + \ell}e^{\frac{1}{12(n - tk - \ell)}}}.
\end{multline}
Just as in the proof of the upper bound, we assume $n - tk - \ell > 0$ and use $n - tk - 1 > n - tk - \ell$ to simplify \eqref{genBallsBinsRestLem2Eq2G} into
\begin{equation}\label{genBallsBinsRestLem2Eq2H}
{{n - tk - 1}\choose{\ell - 1}} \ \geq \ \frac{1}{(\ell - 1)!} \cdot (n - tk - \ell)^{-\ell - 1} \cdot \frac{e^{-n + tk + 1}e^{\frac{1}{12(n - tk - 1)+ 1}}}{e^{-n + tk + \ell}e^{\frac{1}{12(n - tk - \ell)}}},
\end{equation}
which in turn is easily consolidated into
\begin{equation}\label{genBallsBinsRestLem2Eq2I}
{{n - tk - 1}\choose{\ell - 1}} \ \geq \ \frac{1}{(\ell - 1)!} \cdot (n - tk - \ell)^{-\ell - 1}e^{1 - \ell} \cdot \frac{e^{\frac{1}{12(n - tk - 1)+ 1}}}{e^{\frac{1}{12(n - tk - \ell)}}}.
\end{equation}
Next, we utilize the inequalities $\frac{1}{12(n - tk - 1) + 1} \geq \frac{1}{12n - 11}$ and $\frac{1}{12(n - tk - \ell)} \leq \frac{1}{12}$ to conclude
\begin{equation}\label{genBallsBinsRestLem2Eq2J}
{{n - tk - 1}\choose{\ell - 1}} \ \geq \ \frac{1}{(\ell - 1)!} \cdot (n - tk - \ell)^{-\ell - 1}e^{1 - \ell}e^{\frac{1}{12n - 11} - \frac{1}{12}}.
\end{equation}
There is one remaining appearance of $t$. Since the exponent $-\ell - 1$ of $n - tk - \ell$ is negative, we will acquire another bound from below by using $n - tk - \ell < n - \ell$. The result \eqref{genBallsBinsRestLem2Eq2B} follows from this. Moreover, the results \eqref{genBallsBinsRestLem2Eq2C} and \eqref{genBallsBinsRestLem2Eq2D} follow from replicating the proofs of \eqref{genBallsBinsRestLem2Eq2A} and \eqref{genBallsBinsRestLem2Eq2B}, but with replacing $k$ with $k - 1$ and $\al$ with $\be$.
Now, we prove the estimate \eqref{genBallsBinsRestLem2Eq2-1}. We rewrite \eqref{genBallsBinsRestLem2Eq2C} by breaking the sums into two sums each: one over even indices, and the other over odd indices. Precisely,
\begin{multline} \label{genBallsBinsRestLem2Eq2K}
M_{n, \ell, k} \ = \ (-1)^{\al}{{\ell}\choose{\al}}{{n - \al k - 1}\choose{\ell - 1}} - (-1)^{\be}{{\ell}\choose{\be}}{{n - \be(k - 1) - 1}\choose{\ell - 1}} + \\
\sum_{\substack{0 \leq t \leq \al - 1 \\ t \ \text{even}}}{{\ell}\choose{t}}{{n - tk - 1}\choose{\ell - 1}} -  \sum_{\substack{0 \leq t \leq \al - 1 \\ t \ \text{odd}}}{{\ell}\choose{t}}{{n - tk - 1}\choose{\ell - 1}} - \\
\sum_{\substack{0 \leq t \leq \be - 1 \\ t \ \text{even}}}{{\ell}\choose{t}}{{n - t(k - 1) - 1}\choose{\ell - 1}} +  \sum_{\substack{0 \leq t \leq \be - 1 \\ t \ \text{odd}}}{{\ell}\choose{t}}{{n - t(k - 1) - 1}\choose{\ell - 1}}. 
\end{multline}
By our choice of $\al$, it is possible that ${{n - \al k - 1}\choose{\ell - 1}}$ equals $1$ or $0$. Similarly, ${{n - \be(k - 1) - 1}\choose{\ell - 1}}$ equals either $1$ or $0$. Since we are handling inequalities, we will assume those terms possess the largest possible magnitudes. Then, we will use the bounds from above \eqref{genBallsBinsRestLem2Eq2A} and \eqref{genBallsBinsRestLem2Eq2C} to estimate the positive terms in the sums from above; we will use the bounds from below \eqref{genBallsBinsRestLem2Eq2B} and \eqref{genBallsBinsRestLem2Eq2D} to estimate the negative terms in the sums from above. 
\begin{multline} \label{genBallsBinsRestLem2Eq2L}
M_{n, \ell, k} \ \leq \ 2{{\ell}\choose{\al}}{{n - \al k - 1}\choose{\ell - 1}} + 2{{\ell}\choose{\be}}{{n - \be(k - 1) - 1}\choose{\ell - 1}} + \\
\sum_{\substack{0 \leq t \leq \ell \\ t \ \text{even}}}{{\ell}\choose{t}}\left(\frac{1}{(\ell - 1)!}\cdot \frac{(n - 1)^{n - \frac{1}{2}}e^{1 - \ell}e^{\frac{1}{12(n - \ell k - 1)} - \frac{1}{12(n - \ell) + 1}}}{(n - (\al - 1) k - \ell)^{n - \ell k - \ell + \frac{1}{2}}}\right) - \\ \sum_{\substack{0 \leq t \leq \ell \\ t \ \text{odd}}}{{\ell}\choose{t}}\left(\frac{1}{(\ell - 1)!} \cdot (n - \ell)^{-\ell - 1}e^{1 - \ell + \ell k}e^{\frac{1}{12n - 11} - \frac{1}{12}}\right) - \\
\sum_{\substack{0 \leq t \leq \ell \\ t \ \text{even}}}{{\ell}\choose{t}}\left(\frac{1}{(\ell - 1)!} \cdot (n - \ell)^{-\ell - 1}e^{1 - \ell + \ell k}e^{\frac{1}{12n - 11} - \frac{1}{12}}\right) + \\ \sum_{\substack{0 \leq t \leq \ell \\ t \ \text{odd}}}{{\ell}\choose{t}}\left(\frac{1}{(\ell - 1)!}\cdot \frac{(n - 1)^{n - \frac{1}{2}}e^{1 - \ell}e^{\frac{1}{12(n - \ell (k - 1) - 1)} - \frac{1}{12(n - \ell) + 1}}}{(n - (\be - 1)(k - 1) - \ell)^{n - \ell(k - 1) - \ell + \frac{1}{2}}}\right).
\end{multline}
Finally, one can interchange the summation with the factors that are independent of $t$, and invoke Lemma \ref{genBallsBinsRestLem2} to conclude the desired upper bound \eqref{genBallsBinsRestLem2Eq2-1}. 

Similarly, we we will use the bounds from below \eqref{genBallsBinsRestLem2Eq2B} and \eqref{genBallsBinsRestLem2Eq2D} to estimate the positive terms in the sums from below; we will use the bounds from above \eqref{genBallsBinsRestLem2Eq2A} and \eqref{genBallsBinsRestLem2Eq2C} to estimate the negative terms in the sums from below. 
\begin{multline} \label{genBallsBinsRestLem2Eq2M}
M_{n, \ell, k} \ \geq \ -2{{\ell}\choose{\al}}{{n - \al k - 1}\choose{\ell - 1}} - 2{{\ell}\choose{\be}}{{n - \be(k - 1) - 1}\choose{\ell - 1}} + \\
\sum_{\substack{0 \leq t \leq \ell \\ t \ \text{even}}}{{\ell}\choose{t}}\left(\frac{1}{(\ell - 1)!} \cdot (n - \ell)^{-\ell - 1}e^{}e^{\frac{1}{12n - 11} - \frac{1}{12}}\right) - \\
\sum_{\substack{0 \leq t \leq \ell \\ t \ \text{odd}}}{{\ell}\choose{t}}\left(\frac{1}{(\ell - 1)!}\cdot \frac{(n - 1)^{n - \frac{1}{2}}}{(n - (\al - 1) k - \ell)^{n - \ell k - \ell + \frac{1}{2}}}e^{1 - \ell + \ell k}e^{\frac{1}{12(n - \ell k - 1)} - \frac{1}{12(n - \ell) + 1}}\right) - \\
\sum_{\substack{0 \leq t \leq \ell \\ t \ \text{even}}}{{\ell}\choose{t}}\left(\frac{1}{(\ell - 1)!}\cdot \frac{(n - 1)^{n - \frac{1}{2}}e^{1 - \ell + \ell k}e^{\frac{1}{12(n - \ell (k - 1) - 1)} - \frac{1}{12(n - \ell) + 1}}}{(n - (\be - 1)(k - 1) - \ell)^{n - \ell(k - 1) - \ell + \frac{1}{2}}}\right) + \\
\sum_{\substack{0 \leq t \leq \ell \\ t \ \text{odd}}}{{\ell}\choose{t}}\left(\frac{1}{(\ell - 1)!} \cdot (n - \ell)^{-\ell - 1}e^{1 - \ell}e^{\frac{1}{12n - 11} - \frac{1}{12}}\right). 
\end{multline}
Just as before, interchange the summation with the factors that are independent of $t$, and invoke Lemma \ref{genBallsBinsRestLem1}. This procedure will prove \eqref{genBallsBinsRestLem2Eq2-2}, completing the entire proof.
\end{proof}

\begin{remark} The last step in the proof of lemma \ref{genBallsBinsRestLem2} could arguably be made more precise if we broke the lemma into cases based on the parity of $\al$ and $\be$. However, doing so would further obscure the desired result while providing virtually no additional insight.
\end{remark}

\end{document}